\documentclass[a4paper,1pt,leqno]{amsart}

\usepackage{amsmath,amssymb,amsthm}
\usepackage{draftcopy}
\usepackage{hyperref}
\usepackage{graphics}
\usepackage{graphicx}
\usepackage{enumerate}

\usepackage{color}

\newtheorem{theorem}{Theorem}[section]
\newtheorem{proposition}[theorem]{Proposition}
\newtheorem{remark}[theorem]{Remark}
\newtheorem{lemma}[theorem]{Lemma}
\newtheorem{corollary}[theorem]{Corollary}

\def\la{\lambda}

\def\nn{\mathbb{N}}
\def\rr{\mathbb{R}}
\def\to{\rightarrow}
\def\be{\begin{equation}}
\def\ee{\end{equation}}
\def\bes{\begin{equation*}}
\def\ees{\end{equation*}}
\def\q{\quad}
\def\intdoi{\int_{0}^{T}\!\!\!\int_{0}^{L}}
\def\l{\left}
\def\ov{\overline}
\def\r{\right}
\def\intt{\int_{0}^{T}}
\def\intx{\int_{0}^{L}}
\def\bdt{\Big|_{t=0}^{t=T}}
\def\bdx{\Big|_{x=0}^{x=L}}
\newcommand{\dx}{\,\mathrm{d}x}
\newcommand{\dt}{\,\mathrm{d}t}

\begin{document}

\title[Control for  Kuramoto-Sivashinsky equation on star shaped trees]{Null-controllability of the linear Kuramoto-Sivashinsky equation on star-shaped trees}

\author[C. M.  Cazacu]{Cristian M. Cazacu}
\address[C. M. Cazacu]{Faculty of Mathematics and Computer Science \& ICUB,  University of Bucharest\\
14 Academiei Street \\ 010014 Bucharest\\ Romania\\
\hfill\break\indent \and
\hfill\break\indent
Institute of Mathematics ``Simion Stoilow'' of the Romanian Academy. Research Group of the Project PN-III-P4-ID-PCE-2016-0035 \\21 Calea Grivitei Street \\010702 Bucharest \\ Romania\\
}
\email{cristian.cazacu@fmi.unibuc.ro}

\author[L. I. Ignat]{Liviu I. Ignat}
\address[L. I. Ignat]{Institute of Mathematics ``Simion Stoilow'' of the Romanian Academy,
Centre Francophone en Math\'{e}matique
\\21 Calea Grivitei Street \\010702 Bucharest \\ Romania
}

\email{liviu.ignat@gmail.com}
\thanks{ C. C. was partially supported by a Young Researchers Grant awarded by The Research Institute of the University of Bucharest (ICUB) and by the CNCS-UEFISCDI Grant No. PN-III-P4-ID-PCE-2016-0035. L. I. was  partially supported by CNCS-UEFISCDI Grant No.  PN-II-RU-TE-2014-4-0007.
 A. P. was partially supported by CNPq (Brasil) and Agence Universitaire de la Francophonie.}

\author{Ademir F. Pazoto}
\address[A. F. Pazoto]{Instituto de Matem\'atica, Universidade Federal do Rio de
Janeiro, P.O. Box 68530, CEP 21945-970, Rio de Janeiro, RJ,  Brasil}
\email{ademir@im.ufrj.br}

\begin{abstract}
In this paper we treat null-controllability properties for  the linear Kuramoto-Sivashinsky equation on a network with two types of boundary conditions. More precisely,
the equation is considered on a star-shaped tree with Dirichlet and Neumann boundary conditions. By using the moment theory we can derive
null-controllability properties with boundary controls acting on the external vertices of the tree.  In particular, the controllability holds
if the \textit{anti-diffusion} parameter of the  equation does not belong to a critical countable set of real numbers. We point out that the critical set for which the null-controllability fails differs from the first model to the second one.
\end{abstract}

\maketitle



\section{Introduction}

In this paper, we consider two control problems on the same simple network formed by the edges of a tree.
The problem we address here enter in the framework of quantum graphs. The name quantum
graph is used for a graph considered as a one-dimensional singular variety and equipped with
a differential operator. Those quantum graphs are metric spaces which can be written as the
union of finitely many intervals, which are compact or $[0,\infty)$ and any two of these intervals
are either disjoint or intersect only at one of their endpoints.


Our main goal is to study  boundary null-controllability properties for the Kuramoto-Sivashinsky (KS) equation
\be\label{KS}
y_t + \la y_{xx}+y_{xxxx}=0,
\ee
on a star-shaped tree denoted $\Gamma$. More precisely, $\Gamma$ is a simplified  topological graph with $N\geq 2$ edges $e_i$, $i\in \{1, \ldots, N\}$, of the same given length $L>0$ and $N+1$ vertices. Besides, all edges intersect at a unique endpoint which is the interior vertex of the graph.  The mathematical formulation of the control problems that we address on  $\Gamma$ stands for a system of $N$-KS equations on the interval $(0, L)$  coupled through the left endpoint $x=0$ as follows
\be\label{control}
\left\{\begin{array}{ll}
y^k_t+ \la y^k_{xx}+y^k_{xxxx}=0, & (t, x)\in (0, T)\times (0,L)\\[10pt]
y^i(t, 0)=y^j(t, 0), & t\in (0, T), \quad  i,j \in \{1, \ldots, N\}\\[10pt]
\sum_{k=1}^{N}y^k_x(t, 0)=0, & t\in (0, T)\\[10pt]
y^i_{xx}(t, 0)=y^j_{xx}(t, 0), & t\in (0, T), \quad  i, j \in \{1, \ldots, N\}\\[10pt]
\sum_{k=1}^{N}y^k_{xxx}(t, 0)=0, & t\in (0, T)\\[10pt]
y^k(0, x)=y^k_0(x), & x\in (0, L).\\[10pt]
\end{array}\right.
\ee
For system \eqref{control} we study two types of boundary control conditions:
\begin{equation*}\label{model1}
(\textrm{I}): \quad \left\{\begin{array}{cc}
y^k(t, L)=0, & \\[10pt]
 y_x^k(t, L)=u^k(t), & \quad k\in \{1, \ldots, N\},
\end{array}\right.
\end{equation*}
respectively
\begin{equation*}\label{model2}
(\textrm{II}): \quad \left\{\begin{array}{cc}
 y_x^k(t, L)=a^k(t),  &   \\[10pt]
  y_{xxx}^k(t, L)=b^k(t), & \quad k\in \{1, \ldots, N\}.
\end{array}\right.
\end{equation*}
Next in the paper we will refer to \eqref{control}-(I) for system \eqref{control} subject to the boundary conditions (I) and to \eqref{control}-(II) for system \eqref{control} with boundary conditions (II).

In system \eqref{control},  $\lambda$ is a positive constant, the functions $y^k = y^k(t, x)$ are real-valued for any $k\in \{1, \ldots, N\}$,  $t$ denotes the time variable, $x$ denotes the space variable and the subscripts for both $t$ and $x$ indicate partial differentiation with respect to each one. The boundary functions $u^k$,  $a^k$ and $b^k$ are considered as control inputs acting on the external nodes. In model \eqref{control}-(II) we impose two controls to act on the same vertex whereas in model  \eqref{control}-(I)  we only require one control. Our main aims  are to see whether we can force the solutions of system \eqref{control} to have certain properties by choosing appropriate
control inputs. The focus here is on the following null-controllability issue:

\vglue 0.1 cm

{\it Given any finite time $T>0$ and any initial state $y_0=(y^k_0)_{k=1, N}$, can we find proper  control inputs in (I) or (II)  ($u=(u^k)_{k=1, N}$ and $a=(a^k)_{k=1, N}$, $b=(b^k)_{k=1, N}$, respectively) to lead the solution of system \eqref{control} to the zero state, i.e.,
\be\label{zerostate}
y^k(T, x)=0, \q \textrm{ for all } x\in (0, L)\q  \text{and} \q k\in \{1,\ldots,  N\} ?
\ee}
For parabolic control problems, in
general, it is not possible to steer the system to an arbitrary prescribed state. Thus, we do not expect
the exact controllability to be true neither for the KS control system. Our motivation for studying such control systems goes back to the quasilinear KS equation
\be\label{nonlinear ks}
y_t + \la y_{xx}+y_{xxxx}+ yy_x=0,
\ee
which was
derived independently by Kuramoto and Tsuzuki in \cite{kt1,kt2} as a model for phase turbulence in reaction-diffusion systems and by Sivashinsky in \cite{MR0671098,MR0502829} as a model for plane flame propagation.  The real positive number $\lambda$ in \eqref{nonlinear ks} is called the anti-diffusion parameter. This
nonlinear partial differential equation also describes incipient instabilities in a variety of physical and
chemical systems (see, for instance, \cite{chen-chang} and \cite{MR1396472}).

The linear control problem on the interval $(0, 1)$ has been first studied in \cite{MR2556747} considering Dirichlet boundary conditions.
By using the moment theory developed by Fattorini and Russell in \cite{MR0335014}, it was proved that this system is null controllable if two controls
act at one endpoint of the interval. If one control is removed the system is controllable
 if and only if the anti-diffusion parameter $\lambda$ does not belong to the following countable
 set of critical values:
  \be\label{set2}
\mathcal{N}_0:=\left\{\pi^2(m^2 + n^2) : (m,n) \in \mathbb{N}^2, 1\leq m < n,\,m\mbox{ and } n \mbox{ have the same parity} \right\}.
\ee
More precisely, there exists a finite-dimensional space of initial conditions
 that cannot be driven to zero with only one control. We point out that the results in \cite{MR2556747} could be extended to intervals of any length $L>0$ by re-scaling the set $\mathcal{N}_0$ in terms of $L$ accordingly.

  Later on, the boundary controllability to the trajectories
of \eqref{nonlinear ks} was proved in \cite{MR2763563} when two controls act on the left Dirichlet boundary conditions.
We also refer to \cite{cerpa-carreno,cerpa-guzman,
MR3058920,MR3356978,takeo, Mazanti, gugat1, gugat2} for related problems and results
on the subject. Particularly, the strategy used in \cite{takeo} can be regarded as
a possibility to study the nonlinear problem for the systems under consideration in this paper.

To our knowledge, the study of the controllability properties of KS systems in the context of quantum graphs has not been yet addressed in the literature neither for the linear equation \eqref{KS}.
At this respect, the program of this work was carried out for a choice of classical boundary conditions and aims to
establish as a fact that the models under consideration  inherit the interesting controllability properties initially observed
for the KS equation posed on a bounded interval.

In order to present our main results, we introduce the following countable sets
 \bes\label{set}
\mathcal{N}_1:=\left\{\frac{\pi^2}{4L^2}((2m)^2 + (2n)^2): (m,n) \in \mathbb{N}^2, \ 1\leq m < n\,  \right\},
\ees
 \bes\label{set4}
\mathcal{N}_2:=\left\{\frac{\pi^2}{4L^2} (2m)^2:  m\in \nn,\  1\leq m  \right\},
\ees
\bes\label{seet}
\mathcal{N}_3:=\left\{\frac{\pi^2}{4L^2}(2n+1)^2  : n \in \nn,\  n \geq 0 \right\},
\ees
\bes\label{set3}
\mathcal{N}_{4}:=\left\{\frac{\pi^2}{4L^2}\left((2m)^2 + (2n+{1})^2\right) : (m,n) \in \mathbb{N}^2, \ 1\leq m, 0\leq  n \,  \right\},
\ees
 \bes\label{set5}
\mathcal{N}_{odd}:=\left\{\frac{\pi^2}{4L^2}\left((2m+1)^2 + (2n+1)^2\right) : (m, n)\in \nn^2,  \ 0\leq  m< n\right\}.
\ees
To simplify the presentation of our main results let us also introduce the notations
$$\mathcal{N}_{mixt}:=\mathcal{N}_3\cup \mathcal{N}_4.$$
Observe that
\be\label{rel}
\mathcal{N}_{1}\cup \mathcal{N}_{odd} =\frac{\mathcal{N}_0}{4L^2}.
\ee
A priori, for the models we study here each control possesses $N$ components but some of the inputs might not be necessary and could vanish completely. In fact, in the case when the system is null-controllable the goal is to intend the controls to act on a minimal number of components. In both cases the control properties are obtained under some  restrictions on the parameter $\lambda$. 
  Roughly speaking, for problem \eqref{control}-(I) we prove that  when $\lambda\notin \mathcal{N}_0/4L^2$ the null controllability property  holds with $N-1$
control inputs which is the optimal number of controls in this case. When $\lambda\in \mathcal{N}_0/4L^2$ the system is not controllable even if we act with $N$-controls.
For  model \eqref{control}-(II) when $\lambda\notin \mathcal{N}_{mixt}$ the controllability holds with $2N-2$ controls, number which is optimal if $\lambda\in \mathcal{N}_{odd}$.  
The number of controls can be reduced to $2N-3$ if $\lambda\notin \mathcal{N}_{mixt}\cup\mathcal{N}_{odd} $. 

For our purposes we need to work in a rigorous functional framework in which Sobolev spaces play a crucial role, namely $L^2(\Gamma)$ and $H^m(\Gamma)$ for $m\in \mathbb{N}^\star$. By $L^2(\Gamma)$ and $H^m(\Gamma)$ we understand the Hilbert spaces
$$L^2(\Gamma):=\prod_{i=1}^{N} L^2(0, L), \quad H^m(\Gamma):=\prod_{i=1}^{N} H^m(0, L), \quad m\geq 1,$$
endowed with their natural norms:
\[
\|\phi\|_{H^m(\Gamma)}^2=\sum _{k=1}^N \|\phi^k\|_{H^m(0,L)}^2,\quad
\phi=(\phi^1,\dots,\phi^N),\ m\geq 0.
\]

Our main results will be stated accurately in the following.

\begin{theorem}[Null-controllability for model \eqref{control}-(I)]\label{main theorem 1}
Let $T>0$ be fixed and $\lambda \not \in \mathcal{N}_{1}\cup \mathcal{N}_{odd}$. Then
 \begin{enumerate}
 \item For any given edge $e_i$ with $i\in \{1, \ldots, N\}$ and any initial state $y_0=(y_0^k)_{k=1, N} \in L^2(\Gamma)$,    there exist  controls $u=(u^k)_{k=1, N}\in (H^1(0, T))^N$  with $u^{i}(t)= 0$ for all $t\in (0,T)$, such that the solution of system \eqref{control}-(I) satisfies
\begin{equation}\label{nullcontrol1}
y^k(T, x)=0, \textrm{  for all } x\in (0, L) \ \text{and} \ k\in \{1, \ldots, N\}.
\end{equation}
\item  For any given  two edges $e_{i_0}$ and $e_{j_0}$, with $i_0, j_0\in \{1, \ldots, N\}$, there exist an initial state $y_0=(y_0^k)_{k=1, N} \in L^2(\Gamma)$  such that for any control $u=(u^k)_{k=1, N}\in (H^1(0, T))^N$ with $u^{i_0}(t)=u^{j_0}(t)=0$ for all $t\in (0, T)$,  the solution of system \eqref{control}-(I) satisfies
    $$y^{k_0}(T, \cdot)\not  \equiv 0,$$
    for some $k_0\in \{1, \ldots, N\}$.
\end{enumerate}
\end{theorem}
\begin{theorem}[Null-controllability for model \eqref{control}-(II)]\label{main theorem 2}
Let $T>0$ be fixed and $\lambda \not \in  \mathcal{N}_{mixt}$. Then
 \begin{enumerate}
\item  For any given edge $e_i$ with $i\in \{1, \ldots, N\}$ and any initial state $y_0=(y_0^k)_{k=1, N}\in L^2(\Gamma)$,   there exist  controls $a=(a^k)_{k=1, N}, b=(b^k)_{k=1, N} \in (H^1(0, T))^N$ with $a^{i}(t)= b^{i}(t)= 0$ for all $t\in (0, T)$, such that the solution of system \eqref{control}-(II) satisfies
\begin{equation}\label{nullcontrol2}
y^k(T, x)=0, \textrm{  for all } x\in (0, L)\ \text{and}\ k\in \{1, \ldots, N\}.
\end{equation}
\item  If  $\lambda\in \mathcal{N}_{odd}$ then for any given two edges $e_{i_0}$ and $e_{j_0}$, with $i_0, j_0\in \{1, \ldots, N\}$, there exist an initial state $y_0=(y_0^k)_{k=1, N} \in L^2(\Gamma)$ such that for any controls $a=(a^k)_{k=1, N}, b=(b^k)_{k=1, N}\in (H^1(0, T))^N$ with either $a^{i_0}(t)=a^{j_0}(t)=b^{i_0}(t)= 0$ or $b^{i_0}(t)=b^{j_0}(t)=a^{i_0}(t)= 0$ for all $t\in (0, T)$,  the solution of system \eqref{control}-(II) satisfies
    $$y^{k_0}(T, \cdot)\not  \equiv 0,$$
    for some $k_0\in \{1, \ldots, N\}$.
    \item  If $\lambda\notin \mathcal{N}_{odd}$ and $N\geq 3$ for any given two edges $e_{i_0}$ and $e_{j_0}$, with $i_0, j_0\in \{1, \ldots, N\}$ and any initial state $y_0=(y_0^k)_{k=1, N}\in L^2(\Gamma)$, there exist controls $a=(a^k)_{k=1, N}, b=(b^k)_{k=1, N}\in (H^1(0, T))^N$ with either $a^{i_0}(t)=a^{j_0}(t)=b^{i_0}(t)= 0$ or $b^{i_0}(t)=b^{j_0}(t)=a^{i_0}(t)= 0$  for all $t\in (0,T)$,  such that the solution of system \eqref{control}-(II) satisfies
    $$y^k(T, x)=0, \textrm{  for any } x\in (0, L), \quad k\in \{1, \ldots, N\}.
$$

\end{enumerate}
\end{theorem}
\begin{remark}
The first statement in Theorem \ref{main theorem 1} asserts that  for any $\lambda \not \in  \mathcal{N}_{1}\cup \mathcal{N}_{odd}$  system \eqref{control}-(I) can be driven to the zero state acting with controls only on $N-1$ components within the $N$ edges of the tree. The inactive control input could be taken on any edge of the tree independently of the choice of the initial data $y_0$. The second statement in Theorem \ref{main theorem 1} ensures that the minimal number of control inputs needed to control system \eqref{control}-(I) is exactly $N-1$.

Similarly,  Theorem \ref{main theorem 2} says that it is sufficient to act only on $2N-2$ components of the system within the all $2N$ components to ensure the null-controllability of the problem. The second part of Theorem \ref{main theorem 2} represents the optimality  of acting on $2N-2$ components to control any initial data of system \eqref{control}-(II) when $\lambda \in \mathcal{N}_{odd}$. 
However, for some values of $\lambda\notin\mathcal{N}_{mixt}$, more precisely when $\lambda\notin\mathcal{N}_{odd}$, the number of the non-vanishing controls  may be reduced.
The third statement of Theorem \ref{main theorem 2} asserts that in this case we can obtain null-controllability with $2N-3$ active controls. 
\end{remark}

In order to prove Theorems \ref{main theorem 1} and \ref{main theorem 2} we make a careful spectral analysis of the corresponding elliptic  differential operator which allows us to transform the controllability problem into an equivalent moment problem.
The later will be solved combining   the general moment theory developed by Fattorini and Russell \cite{MR0335014} and the asymptotic analysis of the eigenvalues and the corresponding eigenfunctions. One of the main difficulty in applying directly the method in \cite{MR0335014} appears in both systems \eqref{control}-(I) and \eqref{control}-(II) and is represented by the presence of multiple eigenvalues for the corresponding eigenvalue problems as shown with precision in Lemma \ref{eigen_partition} and Lemma \ref{eigen_partition3}.  At that point it is important to note that, for making possible the existence of the controls, the choice of the parameter $\lambda$ plays an important role. It is also worth mentioning that the general theory in  \cite{MR0335014} allows to build solutions to the moment problem in $L^2(0,T)$  but in any space $H^s(0,T)$, $s \geq 0$, and in consequence in $C^\infty(0, T)$.

The extension of the above results to the case when the edges of the star shaped tree have different lengths needs a different approach. Following \cite{MR1769018}, other coupling conditions may be imposed at the internal node. The analysis of the controllability properties for the KS system with other coupling conditions at $x=0$ remains to be considered elsewhere.

The analysis described above is organized in two sections: in Section \ref{secI} we study problem \eqref{control}-(I) and prove Theorem \ref{main theorem 1}. Section \ref{secII} is devoted to problem \eqref{control}-(II)
and the proof of Theorem \ref{main theorem 2}.
The main effort we put in this paper concerns the proofs of Theorem \ref{main theorem 1} and Theorem \ref{main theorem 2} where we focus to obtain control results with minimal number of active control inputs. On the other hand we may also wonder in which conditions we can obtain control properties with no restrictions on the number of the control components. We answer to these questions  in section \ref{secf}.
Besides, in view of the spectral results in sections \ref{prelim1} and \ref{prelim2} we are able to obtain new control results for the linear KS on an interval which were not analyzed in \cite{MR2556747}. These aspects will be detailed in section \ref{secf2}.

\section{The KS equation of type (I)}\label{secI}
The controllability problem \eqref{control}-(I) will be studied by using the method of moments due to Fattorini and Russell \cite{MR0335014}. Therefore, a careful spectral analysis of the involved elliptic operator is necessary.
\subsection{Spectral analysis}\label{spectanal}
For any $\lambda>0$ let us consider the following spectral problem on $\Gamma$:
\be\label{spect}
\left\{\begin{array}{ll}
 \la \phi^k_{xx}+\phi^k_{xxxx}=\sigma \phi^k, & x\in (0,L)\\[10pt]
\phi^i(0)=\phi^j(0), &  \quad  i,j \in \{1, \ldots, N\}\\[10pt]
\phi^k(L)=\phi_x^k(L)=0, & k\in \{1, \ldots, N\}\\[10pt]
\sum_{k=1}^{N}\phi^k_x(0)=0, & \\[10pt]
\phi^i_{xx}(0)=\phi^j_{xx}(0), & \quad  i, j \in \{1, \ldots, N\}\\[10pt]
\sum_{k=1}^{N}\phi^k_{xxx}(0)=0. & \\[10pt]
\end{array}\right.
\ee

To begin with, we introduce the fourth order operator
$$A:  D(A)\subset L^2(\Gamma)\rightarrow L^2(\Gamma)$$ given by
\be\label{op}
\l\{\begin{array}{ll}
      A\phi^k=\la \phi^k_{xx}+\phi^k_{xxxx},\quad k \in \{1, \ldots, N\}  \\[10pt]
      D(A)=\l\{\begin{array}{ll} \phi=(\phi^k)_{k=1, N} \in H^4(\Gamma) \ | \ \phi^k(L)=\phi_x^k(L)=0, \quad k \in \{1, \ldots, N\}, \\[5pt] \phi^i(0)=\phi^j(0), \  \phi^i_{xx}(0)=\phi^j_{xx}(0), \quad i, j \in \{1, \ldots, N\},\\[5pt]
                 \sum_{k=1}^{N}\phi^k_x(0)=0, \          \sum_{k=1}^{N}\phi^k_{xxx}(0)=0
               \end{array}\r\}.
    \end{array}\r.
\ee
Remark that spectral problem \eqref{spect} is equivalent to
\be\label{equiv}
\l\{\begin{array}{ll}
      A\phi= \sigma \phi,\\[5pt]
      \phi \in D(A).
      \end{array}\r.
\ee
To study  this eigenvalue problem  we firstly claim that
 \begin{proposition}\label{p1}
 For any $\mu> \la^2/4$ the operator
 $$A+\mu I :D(A)\subset L^2(\Gamma)\rightarrow L^2(\Gamma),$$
 is a non-negative self-adjoint operator with compact inverse. In particular, it has a pure discrete spectrum consisted by a sequence of nonnegative eigenvalues $\{\sigma_{\mu, n}\}_{n\in \nn}$ satisfying $\lim_{n\to \infty}\sigma_{\mu, n}=\infty$.  Moreover, up to a normalization, the corresponding eigenfunctions $\{\phi_{\mu, n}\}_{n\in \nn}$ form an orthonormal basis of $L^2(\Gamma)$.
 \end{proposition}

\proof

Firstly, it is easy to observe that $D(A)$ is dense in $L^2(\Gamma)$. Next the proof will be done in several steps.

\noindent \textit{Step 1:  $A$ is a symmetric operator.}  Indeed, after integrations by parts, we get
\begin{align}
(Au, v)_{L^2(\Gamma)}&=(u, Av)_{L^2(\Gamma)}\nonumber\\
&=\sum_{k=1}^{N} \intx (u^k_{xx}v^k_{xx}-\la u^k_xv^k_x) \dx, \quad \forall u, v \in D(A).\nonumber
\end{align}

\noindent\textit{Step 2:  $A+\mu I$ is maximal monotone for any $\mu> \la^2/4$.}
Firstly, let us show the monotonicity property
\[
((A+\mu I)u, u)_{L^2(\Gamma)}\geq  0, \quad \forall u\in D(A).
\]
Indeed,
\be\label{pr1}
((A+\mu I )u, u)_{L^2(\Gamma)}=\sum_{k=1}^{N} \intx (|u^k_{xx}|^2+\mu |u^k|^2-\la |u^k_x|^2) \dx.
\ee
On the other hand for $u\in D(A)$ we have
\begin{align*}
\sum_{k=1}^{N} \intx u^k_{xx}u^k \dx &=\sum_{k=1}^{N } u^k_x u^k \bdx -\sum_{k=1}^{N} \intx |u^k_x|^2 \dx\\
&=   -\sum_{k=1}^{N} \intx |u^k_x|^2 \dx.
\end{align*}
 Therefore, from  the inequality of arithmetic and geometric means it holds
 \be\label{pr2}
 \sum_{k=1}^{N} \intx |u^k_x|^2 \dx \leq \sum_{k=1}^{N} \intx \l(\frac{1}{\la} |u^k_{xx}|^2+ \frac{\la}{4}|u^k|^2 \r)\dx.
 \ee
 Combining \eqref{pr1} and \eqref{pr2} we obtain
 $$((A+\mu I)u, u)_{L^2(\Gamma)}\geq \l(\mu -\frac{\la^2}{4}\r)\sum_{k=1}^{N} \intx |u^k|^2\dx, $$
which is nonnegative for $\mu> \la^2 /4$.

  Next we emphasize that  $A+\mu I$ is maximal, i.e., for any  $f \in L^2(\Gamma)$  there exists a unique $u\in D(A)$ such that
$
 (A+\mu I)u=f.
$
To do that, first we  consider the variational formulation
\be\label{var}
\l\{\begin{array}{l}
  a(u, v)=(f, v)_{L^2(\Gamma)}, \quad \forall v\in V \\
  u\in V,
\end{array}\right.
\ee
where $V$ denotes the Hilbert space
\[
V=\l\{\begin{array}{ll} \phi=(\phi^k)_{k=1, N} \in H^2(\Gamma) \ | \ \phi^k(L)=\phi_x^k(L)=0, \quad k \in \{1, \ldots, N\}, \\[5pt] \phi^i(0)=\phi^j(0), \ i, j \in \{1, \ldots, N\},\ \sum_{k=1}^{N}\phi^k_x(0)=0 \\
          \end{array}\r\}
\] endowed with the $H^2(\Gamma)$-norm
and $a(\cdot, \cdot):V\times V\to \rr$ denotes the bilinear form
$$a(u, v)=\sum_{k=1}^{N}\intx \l( u^k_{xx}v^k_{xx} -\la u^k_x v^k_x +\mu u^k v^k \r)\dx.  $$
It is easy to see that $a(\cdot, \cdot)$ is symmetric and continuous. In addition, it is also coercive. Indeed, let $\delta>0$ small enough such that $\mu> \la^2/(4-4\delta)$. Then, arguing as in \eqref{pr2} we obtain
 \[
 \sum_{k=1}^{N} \intx |u^k_x|^2 \dx \leq \sum_{k=1}^{N} \intx \l(\frac{1-\delta}{\la} |u^k_{xx}|^2+ \frac{\la}{4(1-\delta)}|u^k|^2 \r)\dx,
 \]
 which together with \eqref{pr1} leads to
$$a(u,u)\geq \delta \sum_{k=1}^{N} \intx |u^k_{xx}|^2\dx+ \l(\mu -\frac{\la^2}{4-4\delta}\r)\sum_{k=1}^{N} \intx |u^k|^2\dx,  $$
and so $a$ is coercive. Applying Lax-Milgram lemma we ensure the existence of a unique $u \in V$ satisfying the variational problem \eqref{var}. In order to justify the maximality of $A+\mu I$ it is sufficient to show that the solution $u$ of \eqref{var} belongs actually to $D(A)$. For that, we refer to the classical regularity arguments for elliptic operators (see, for instance, \cite{MR697382}).

\noindent \textit{Step 3: $A$ is a self-adjoint operator with compact inverse.}
 Since $A+\mu I :D(A)\subset L^2(\Gamma)\to L^2(\Gamma)$ is a symmetric operator and maximal monotone, it is a self-adjoint operator (see, e.g., \cite{MR697382}).

Moreover, we have that the linear operator $(A+\mu I)^{-1}:L^2(\Gamma)\to L^2(\Gamma)$, given by
$$(A+\mu I)^{-1} f=u \in D(A)\subset L^2(\Gamma),$$
satisfies
$$\| (A+\mu I)^{-1} f \|_{H^2(\Gamma)}\leq C \|f\|_{L^2(\Gamma)},$$
for some positive constant $C>0$.  Since the embedding $H^2(\Gamma)\hookrightarrow L^2(\Gamma)$ is compact it follows that
$(A+\mu I)^{-1}$ is a compact operator. Then applying the classical spectral results for compact self-adjoint operators we conclude the proof of Proposition \ref{p1}.
\endproof

 \begin{remark}
The spectrum $\{\sigma_n\}_{n\in \nn}$ of problem \eqref{equiv} is obtained by shifting the spectrum of $A+\mu I$ in Proposition \ref{p1}, i.e.,
$$\sigma_n:=\sigma_{\mu, n}-\mu,  \quad \forall n\in \nn,$$
with the corresponding eigenfunctions
$$\phi_{n}:=\phi_{\mu, n}, \quad \|\phi_n\|_{L^2(\Gamma)}=1, \quad \forall n\in \nn.$$
In particular, it holds that
\begin{equation}\label{spectrumA}
-\frac{\lambda^2}{4}\leq \sigma_n, \ \forall n\in \nn, \quad \quad \sigma_n \rightarrow \infty, \textrm { as }n\rightarrow \infty.
\end{equation}
 \end{remark}
 The following proposition will be also very useful in our analysis.
 \begin{proposition}\label{propp}
 The lower bound $-\lambda^2 /4$ is not an eigenvalue for the operator $A$.
  \end{proposition}
   \proof If $(-\lambda^2/4, \phi)$ is an eigenpair for $A$, then $A\phi=\lambda^2/4\phi$ for some nontrivial $\phi=(\phi^k)_{k=1, N}\in D(A)$. Then, integration by parts leads to
\[
0=\left(\left(A+\frac{\lambda^2}{4}\right)\phi, \phi\right)_{L^2(\Gamma)}=\sum_{k=1}^{N} \int_0^L \left(\phi_{xx}^k+\frac{\lambda}{2}\phi^k\right)^2 dx.
\]
Therefore, for any $k\in \{1, \ldots, N\}$ the component $\phi^k$ must satisfy the equation $\phi_{xx}^k+\frac{\lambda}{2}\phi^k=0$ in $(0, L)$ subject to the boundary conditions $\phi^k(L)=\phi_x^k(L)=0$. These conditions allow us to obtain $\phi^k\equiv 0$, for all $k\in \{1, \ldots, N\}$ which is in contradiction with $\phi\not \equiv 0$.
\endproof

\subsection{Qualitative properties of the eigenvalues}\label{spectqual}

The main goal of this section is to provide an asymptotic formula for the behavior of eigenvalues of system \eqref{spect}. Particularly, this result will play an important role to prove the null-controllability of problem \eqref{control}-(I).

For any fixed $\lambda>0$  let us firstly consider the following two eigenvalue problems on the interval $(0, L)$:
\begin{equation}\label{eign.1.tree}
\left\{
\begin{array}{ll}
\lambda \Psi_{xx}+\Psi_{xxxx}=\sigma \Psi, & x\in (0,L),\\[10pt]
 \Psi_x(0)=0, \ \Psi_{xxx}(0)=0 , &\\[10pt]
 \Psi(L)=\Psi_x(L)=0 &\\[10pt]
\end{array}
\right.
\end{equation}
and
\begin{equation}\label{eign2}
\left\{
\begin{array}{ll}
\lambda \Phi_{xx}+\Phi_{xxxx}=\sigma \Phi, & x\in (0,L),\\[10pt]
 \Phi(0)=0, \ \Phi_{xx}(0)=0 , &\\[10pt]
 \Phi(L)=\Phi_{x}(L)=0, &\\[10pt]
\end{array}
\right.
\end{equation}
respectively. As in Section \ref{spectanal} we can easily show that both  systems \eqref{eign.1.tree} and \eqref{eign2} possess a sequence of eigenvalues $\{\sigma_n\}_{n\geq 0}$ which tends to infinity and is strictly bounded from below by $-\lambda^2/4$.  Before going through let us fix some notations which will be useful in the forthcoming sections. For any $\lambda>0$  we denote
\be\label{notations}
\left\{\begin{array}{ll}
 \alpha:=\sqrt{\frac{-\lambda+\sqrt{\lambda^2 +4\sigma}}{2}}, \quad   \beta:=\sqrt{\frac{\lambda+\sqrt{\lambda^2 +4\sigma}}{2}}, & \textrm{ if } \sigma\geq 0\\[10pt]
  \gamma:=\sqrt{\frac{\lambda-\sqrt{\lambda^2 +4\sigma}}{2}} \quad \beta:=\sqrt{\frac{\lambda+\sqrt{\lambda^2 +4\sigma}}{2}}, & \textrm{ if } -\frac{\lambda^2}{4}< \sigma<0
\end{array}\right.
\ee
for which we have the relations
\be\label{connection}
\left\{\begin{array}{ll}
   \beta^2-\alpha^2 =\lambda, & \textrm{ if } \sigma \geq 0  \\[10pt]
  \beta^2+\gamma^2=\lambda, &  \textrm{  if } -\frac{\lambda^2}{4}< \sigma< 0
\end{array}\right.
\ee
and
\be\label{sigma}
\sigma=\left\{\begin{array}{ll}
\alpha^2 \beta^2, & \textrm{ if } \sigma \geq 0\\[10pt]
-\beta^2 \gamma^2, & \textrm{  if } -\frac{\lambda^2}{4}< \sigma \leq 0.
\end{array}\right.
\ee

Coming back to our spectral problem \eqref{spect},  we introduce the functions
\begin{equation}\label{functie_suma}
S:=\sum_{k=1}^{N} \phi^k
\end{equation}
and
\begin{equation}\label{functie_diferenta}
D^k:=\phi^k - \frac{S}{N},  \quad k\in \{1, \ldots, N\}.
\end{equation}
The motivation for analyzing systems \eqref{eign.1.tree} and \eqref{eign2} is due to the fact that $S$ verifies \eqref{eign.1.tree} whereas $D^k$ satisfies \eqref{eign2} for all $k\in \{1, \ldots, N\}$.

Next we state and prove some preliminary results.

\begin{lemma}\label{not.common}
	For any $\lambda>0$ the eigenvalue problems \eqref{eign.1.tree} and \eqref{eign2} have no any common eigenvalue $\sigma$.  In addition, any eigenvalue of either \eqref{eign.1.tree} or \eqref{eign2} is simple.

Moreover, if $\la  \not \in \mathcal{N}_0/4L^2$ then any  eigenfunction $\phi$ of either \eqref{eign.1.tree} or \eqref{eign2} satisfies $\phi_{xx}(L)\neq 0.$
\end{lemma}

\begin{proof}
First we show that problems \eqref{eign.1.tree} and \eqref{eign2} have no common eigenvalues.
Let us assume that there exists $\sigma$ and two functions $\Psi$ and $\Phi$ not identically vanishing,  satisfying  \eqref{eign.1.tree} and \eqref{eign2}, respectively.

The boundary conditions at $x=0$ in \eqref{eign.1.tree} and \eqref{eign2} allow to introduce the even and odd extensions  of $\Psi$, respectively $\Phi$,  with respect to $x=0$. More precisely, we consider
\[
 \ov{\Psi}(x):=\l\{\begin{array}{ll}
               \Psi(x), & x\in [0, L]\\
               \Psi(-x), & x\in [-L, 0]
             \end{array}\r.
\]
and
\[
\ov{\Phi}(x):=\l\{\begin{array}{ll}
               \Phi(x), & x\in [0, L]\\
               -\Phi(-x), & x\in [-L, 0].
             \end{array}\r.
\]

Then $\ov{\Psi}$ and $\ov{\Phi}$ verify
\begin{equation}\label{eign4}
\left\{
\begin{array}{ll}
\lambda \ov{\Psi}_{xx}+\ov{\Psi}_{xxxx}=\sigma \ov{\Psi}, & x\in (-L,L)\\[10pt]
 \ov{\Psi}(-L)=\ov{\Psi}(L)=\ov{\Psi}_x(-L)=\ov{\Psi}_x(L)=0 &
\end{array}
\right.
\end{equation}
and
\begin{equation}\label{eign3}
\left\{
\begin{array}{ll}
\lambda \ov{\Phi}_{xx}+\ov{\Phi}_{xxxx}=\sigma \ov{\Phi}, & x\in (-L,L)\\[10pt]
 \ov{\Phi}(-L)=\ov{\Phi}(L)=\ov{\Phi}_x(-L)=\ov{\Phi}_x(L)=0, &
\end{array}
\right.
 \end{equation}
 respectively. Finally, let us  denote
$$\hat{\Psi}(y):=\ov{\Psi}(2L y-L), \quad \hat{\Phi}(y):=\ov{\Phi}(2L y-L), \quad y\in (0,1).$$
In view of \eqref{eign4} and \eqref{eign3} it follows that $\hat \Psi$ and $\hat \Phi$ satisfy the same eigenvalue problem
\begin{equation}\label{eign5}
\left\{
\begin{array}{ll}
\lambda (2L)^2 \phi_{xx}+\phi_{yyyy}=\sigma (2L)^4 \phi, & y\in (0,1)\\[10pt]
 \phi(0)=\phi(1)=\phi_y(0)=\phi_y(1)=0. &
\end{array}
\right.
\end{equation}
The arguments in \cite{MR2556747} show that problem \eqref{eign5} admits simple eigenvalues. This means that $\hat \Phi=\alpha \hat \Psi$ for some constant $\alpha\neq 0$ and, equivalently, we have $\ov\Phi=\alpha \ov\Psi$. Since $\ov\Psi$ is an even function and $\ov\Phi$ is an odd function (both of them vanishing on the  boundary) we necessarily have $\ov\Psi=\ov\Phi\equiv 0$, which contradicts our assumption. Then the first part of lemma is proved.

The fact that the eigenvalues of both \eqref{eign.1.tree} and \eqref{eign2} are simple is a consequence of the construction above. Indeed, if $(\sigma, \Psi_1)$ and $(\sigma, \Psi_2)$  are eigenpairs for \eqref{eign.1.tree} we get that
$(\sigma, \hat{\Psi}_1)$ and $(\sigma, \hat{\Psi}_2)$ are eigenpairs of \eqref{eign5}. Since any eigenvalue of \eqref{eign5} is simple there exists a constant $c\neq 0$ such that $\hat{\Psi}_1(y)=c\hat{\Psi}_2(y)$ for all $y\in [0,1]$. This is equivalent to $\overline{\Psi_1}(x)=c\overline{\Psi_2}(x)$ for all $x\in [-L,L]$ and therefore $\Psi_1(x)=c \Psi_2(x)$ for any $x\in [0,L]$. Hence $\sigma$ is a simple eigenvalue for \eqref{eign.1.tree}. With a similar  argument we get that each eigenvalue of  \eqref{eign2} is also simple.

For the last part of the proof let us assume that $(\sigma, \Psi)$ is an eigenpair of \eqref{eign.1.tree}. With the notations above,
   applying \cite[Lemma 2.1]{MR2556747} for \eqref{eign5} under the assumption $\la (2L)^2 \not \in \mathcal{N}_0$  it follows  $\hat{\Psi}_{xx}(0)\neq 0$. Due to the symmetry of the boundary conditions we notice that the function $x\mapsto \hat \Psi(1-x)$ is also an eigenfunction of problem \eqref{eign5} and therefore we also have $\hat \Psi_{xx}(1)\neq 0$. This gives us the desired property for $\Psi$, $\Psi_{xx}(L)\neq 0$. Analogously, it follows that $\Phi_{xx}(L)\neq 0$ for any eigenfunction of \eqref{eign2}. The proof of Lemma \ref{not.common} is finished.
\end{proof}

As a consequence of Lemma \ref{not.common} we have the following partition for the eigenvalues of system \eqref{spect}.
\begin{lemma}\label{eigen_partition}
For any given $\lambda>0$, $\sigma$ is an eigenvalue for system \eqref{spect} if and only if $\sigma$ is an eigenvalue for either \eqref{eign.1.tree} or \eqref{eign2}. More precisely
\begin{enumerate}
\item If $(\sigma, \phi=(\phi^k)_{k=1,N})$ is an eigenpair of \eqref{spect} we have the alternative:
    \begin{enumerate}
    \item\label{1.a} There exists $\Psi$ with   $(\sigma,\Psi)$
     eigenpair of  \eqref{eign.1.tree} and
     $$\phi=(\Psi, \ldots, \Psi),$$
    or
    \item\label{1.b}  There exists $\Phi$ with $(\sigma, \Phi)$  eigenpair of \eqref{eign2} and there exist  constants $c_1, \ldots, c_N$ (not all vanishing) with $\sum_{k=1}^{N}c_k=0$ such that
        $$\phi=(c_1 \Phi, \ldots, c_N\Phi).$$
  \end{enumerate}
\item\label{doi} Conversely, if $(\sigma, \Psi)$ is an eigenpair of \eqref{eign.1.tree} then $(\sigma, \phi=(\Psi, \ldots, \Psi, \Psi))$ is an eigenpair of \eqref{spect} whereas if $(\sigma, \Phi)$ is an eigenpair of \eqref{eign2} then, for any constants $c_k$ (not all vanishing) such that $\sum_{k=1}^{N}c_k=0$, it holds that $(\sigma, \phi=(c_1\Phi, \ldots, c_N\Phi))$ is also eigenpair of \eqref{spect}.
\end{enumerate}
\end{lemma}

\begin{proof} The ``if" implication is a consequence of \eqref{doi} whose proof is trivial. Let us prove  the ``only if" implication. For that, let us assume that $(\sigma, \phi=(\phi^k)_{k=1, N})$ is an eigenpair of \eqref{spect}. Then $\sigma$ verifies \eqref{eign.1.tree} for $\Psi=S$ in \eqref{functie_suma}. In addition, $\sigma$ verifies \eqref{eign2} for any $\Phi=D^k$ in \eqref{functie_diferenta}.
By Lemma \ref{not.common} we distinguish two cases: either $S\not\equiv 0$ and  $D^k\equiv 0$ for all $k\in \{1,\dots, N\}$ or  $S\equiv 0$.

If $S\not\equiv 0$ then ($\sigma$, $S$) is an eigenpair for \eqref{eign.1.tree} and  $D^k\equiv 0$ for all $k\in \{1,\dots, N\}$. Then we obtain $\phi=(\Psi,\dots,\Psi)$ where $\Psi=S/N$ is an  eigenfunction for problem \eqref{eign.1.tree} and \eqref{1.a} is proved.

 Otherwise, if $S\equiv0$ then $D^k=\phi^k$ for all $k\in \{1, \ldots, N\}$. From the initial assumption  there exists $k_0 \in \{1, \ldots, N\}$ such that $\Phi:=\phi^{k_0}\neq 0$, and therefore,  $(\sigma, \Phi)$ is an eigenpair for \eqref{eign2}.
We know that  $(\sigma, \phi^k)$ also verifies  \eqref{eign2} for any $k=1,\dots, N$. In view of Lemma \ref{not.common} the eigenspace of $\sigma$ in \eqref{eign2} has dimension 1. Thus,  there exist some constants $c_k$ such that  $\phi^k=c_k\Phi$ for any $k\in \{1, \ldots, N\}$. Then, since $S\equiv 0$ we get $\sum_{k=1}^{N}c_k=0$. This proves \eqref{1.b}   which completes the proof of Lemma \ref{eigen_partition}.
\end{proof}

\begin{lemma}\label{neqzero}
For any $\la\not \in \mathcal{N}_0/4L^2$ any eigenfunction $\phi=(\phi^k)_{k=1, N}$ of \eqref{spect} satisfies
$\phi^k_{xx}(L) \neq 0$ for at least two indexes $k\in \{1, \ldots, N\}$.
\end{lemma}
\proof This is a trivial consequence of Lemma \ref{not.common} and Lemma \ref{eigen_partition}.
\endproof

\begin{lemma}\label{asymptbehavior}
	The positive  eigenvalues $\{\sigma_n\}_{n\geq n_0}$ of problem \eqref{spect}  can be partitioned into
\begin{equation}
\label{part.1}
  \{\sigma_n \ | \ n \geq n_0\}:=\{\sigma_{1, n} \ | \ n\geq 0\}\cup \{\sigma_{2, n} \ | \ n\geq 0\},
\end{equation}
where
$\{\sigma_{1, n}\}_{n\geq 0}$ are  the different  eigenvalues of \eqref{eign.1.tree} each of them being simple, 
whereas $\{\sigma_{2, n}\}_{n\geq 0}$ are the different eigenvalues of \eqref{eign2} each of them having multiplicity $N-1$.
Moreover,  they satisfy the following asymptotic properties:
	\begin{equation*}
  \sigma_{1, n}= \left(\frac \pi{L}\right)^4\left(n-\frac{1}{4}\right)^4 + o(n^3), \quad n\rightarrow \infty
\end{equation*}
and
\begin{equation*}
  \sigma_{2, n}= \left(\frac \pi{L}\right)^4\left(n+\frac{1}{4}\right)^4 + o(n^3), \quad n\rightarrow \infty.
\end{equation*}
Therefore,
\begin{equation*}
  \sigma_n=\left(\frac{\pi}{2L}\right)^4\left(n-n_0-\frac{1}{2}\right)^4+o(n^3), \quad n \rightarrow \infty.
\end{equation*}
\end{lemma}
\begin{proof} Partition \eqref{part.1} and the multiplicity of the eigenvalues are  consequences of Lemma \ref{eigen_partition}. 
Since there exists a finite number of non-positive eigenvalues we just concentrate on the positive eigenvalues $\sigma$ when we analyze their asymptotic behaviour.

	Let us consider $(\sigma, \phi) $ an eigenpair of \eqref{spect} with $\sigma>0$. In view of Lemma \ref{eigen_partition} we distinguish two cases. \\
\textit{Case 1.} We have $\phi=(\Psi, \ldots, \Psi)$ where $(\sigma, \Psi)$ is an eigenpair  for \eqref{eign.1.tree}.
 Since $\sigma>0$, with the notations in \eqref{notations} we have
\be\label{cucbau}
\Psi(x)=C_1 \cos (\beta x)+ C_2 \sin (\beta x)+ C_3 \cosh (\alpha x)+ C_4 \sinh (\alpha x),
\ee
where $C_1, C_2, C_3, C_4$ are such $\Psi$  satisfies the boundary conditions in \eqref{eign.1.tree} and $\Psi\not \equiv 0$. From the conditions at  $x=0$ we get $C_2=C_4=0$. From the conditions at $x=L$ we obtain $\Psi\not \equiv 0$  when the compatibility conditions
\be\label{eigbau}
  \beta \cosh(\alpha L)\sin(\beta L)+\alpha \sinh(\alpha L)\cos(\beta L)=0, \quad \alpha, \beta >0,
\ee
are satisfied.  This is equivalent to
$$\beta \tan (\beta L)=-\alpha \tanh (\alpha L), \quad \cos (\beta L)\neq 0$$
or equivalently (in view of \eqref{notations})
\begin{equation}\label{estim}
  \sqrt{\lambda +\alpha^2} \tan (\sqrt{\lambda+\alpha^2}L)=-\alpha \tanh (\alpha L), \quad \cos (\sqrt{\lambda+\alpha^2} L) \neq 0.
\end{equation}

It is not difficult to note that the function $(0, \infty)\ni \alpha \mapsto -\alpha  \tanh ( \alpha L)$ is strictly decreasing whereas the function $(0,\infty)\setminus \{\alpha \ |  \sqrt{\lambda+\alpha^2}L =n\pi+\pi/2 \ ,  n\in \nn\}\ni \alpha \mapsto \sqrt{\lambda+\alpha^2} \tan (\sqrt{\lambda +\alpha^2}L)$ is increasing on each interval of the domain  (for that to be proved we just have to look at the sign of the corresponding derivatives, see also \cite{MR1947955} for similar arguments). So, there exists two strictly increasing  sequences  $\{\alpha_{1, n}\}_{n\geq 0}$ and  $\{\beta_{1, n}\}_{n\geq 0}$ with $\beta_{1, n}^2=\lambda +\alpha_{1, n}^2$,  solutions for \eqref{eigbau}, where $\beta_{1, n}L\in (n\pi-\pi/2,n\pi+\pi/2)$ .
Then the sequence of positive eigenvalues  $\{\sigma_{1, n}\}_{n\geq 0}$ is given by
\be\label{cucubau}
  \sigma_{1, n}=\beta_{1, n}^2 \alpha_{1, n}^2.
\ee
Coming back to \eqref{estim} we have
$$-\frac{\alpha_{1, n}}{\sqrt{\lambda+\alpha_{1, n}^2}}=
\frac{\tan\left(\sqrt{\lambda+\alpha_{1, n}^2}L\right)}{\tanh (\alpha_{1, n} L)}.$$
Passing to the limit  we obtain
$$\tan (\beta_{1, n} L)=\tan\left(\sqrt{\lambda+\alpha_{1, n}^2}L\right) \rightarrow - 1, \quad n \rightarrow \infty.$$
Then we get
\begin{equation}
\label{formula.betan}
  \beta_{1, n} = \frac{n\pi}{L}-\frac{\pi}{4L}+o(1), \quad \textrm{ as } n\rightarrow \infty,
\end{equation}

which combined with \eqref{cucubau} gives the asymptotic behavior of $\{\sigma_{1, n}\}_{n\geq 0}$ in the present lemma.

\textit{Case 2}.  We have $\phi=(c_1 \Phi, \ldots, c_N\Phi)$ for some constants $c_k$ with $\sum_{k=1}^{N}c_k=0$, where $(\sigma, \Phi)$ is an eigenpair of \eqref{eign2}.
 The requirement for $\sigma$ to be an eigenvalue of \eqref{eign2} is equivalent to
  \begin{equation}\label{cucubau1}
  \beta \sinh(\alpha L)\cos(\beta L)-\alpha \cosh(\alpha L)\sin(\beta L)=0
  	  \end{equation}
or equivalently
\be\label{cucubau2}
\frac{1}{\alpha} \tanh (\alpha L)=\frac{1}{\beta}\tan(\beta L), \quad \textrm{ with } \beta=\sqrt{\lambda+\alpha^2},  \quad \cos (\beta L)\neq 0.
\ee
 In the same way as in the previous case it is not difficult to prove that the function $(0, \infty)\ni \alpha \mapsto \tanh(\alpha L)/\alpha$ is decreasing whereas the function $(0, \infty)\setminus \{\alpha \ | \ \sqrt{\lambda +\alpha^2 }L =n \pi +\pi/2, n\in \nn  \}\ni \alpha \mapsto \tan (\sqrt{\lambda+\alpha^2}L)/(\sqrt{\lambda+\alpha^2}L)$ is strictly increasing on each open interval of the domain. Thus we obtain two sequences  $\{\alpha_{2, n}\}_{n\geq 0}$ and $\{\beta_{2, n}\}_{n\geq 0}$ satisfying \eqref{cucubau2}   with $\beta_{2,n}L\in  (n\pi-\pi/2,n\pi+\pi/2)$. Rewriting \eqref{cucubau1} as
 $$\frac{\sqrt{\lambda+\alpha_{2, n}^2}}{\alpha_{2, n}}=
\frac{\tan\left(\sqrt{\lambda+\alpha_{2, n}^2}L\right)}{\tanh (\alpha_{2, n} L)}$$
and passing to the limit as $n$ tends to infinity we obtain similarly as in Case 1 that
$$\beta_{2, n} = \frac{n\pi}{L}+\frac{\pi}{4L}+o(1), \quad \textrm{ as } n\rightarrow \infty.$$
Since $\sigma_{2, n}= \beta_{2, n}^2 \alpha_{2, n}^2$ we obtain the asymptotic behaviour. 

On the other hand let us observe that
  $$\sigma_{1, n}< \sigma_{2, n}< \sigma_{1, n+1}< \sigma_{2, n+1}, \quad \forall n\in \nn.$$ 
 Let $n_0$ be such that $\sigma_{n_0}$ is the first positive eigenvalue of \eqref{spect}. 
  
  By  concatenating the sequences $\{\sigma_{1, n}\}_{n\geq 0}$ and $\{\sigma_{2, n}\}_{n\geq 0}$ we obtain $\sigma_{n_0+2n}=\sigma_{1, n}$ and $\sigma_{n_0+2n+1}=\sigma_{2, n}$ for all $n\geq 0$ which implies the asymptotic behavior of the sequence $\{\sigma_n\}_{n \geq n_0}$  Thus the present lemma is proved.
\end{proof}
As a consequence of Lemma \ref{asymptbehavior} we easily obtain the spectral gap
 \begin{corollary}\label{gapp}
   Let $\{\sigma_n\}_{n\geq n_0}$   the increasing set of positive eigenvalues of problem \eqref{spect}. 
   There exists a constant $c_{gap}>0$
   such that
     $$\sigma_{n+1}-\sigma_n> c_{gap}, \quad \forall n\geq n_0.$$
 \end{corollary}

For our next purposes we also need to prove some asymptotic properties for the eigenfunctions evaluated at $x=L$ as follows.

\begin{lemma}\label{asy_eigenfunct}
For any $n\geq n_0$ let $\phi_n=(\phi_n^k)_{k=1, N}$ be an eigenfunction of problem \eqref{spect} with $\|\phi_n\|_{L^2(\Gamma)}=1$ corresponding to the positive eigenvalue $\sigma_n$.
There exist the constants $a, b>0$ depending only on $L$ and $N$ such that for all $ n\geq n_0$
\be\label{assy1}
 a n^2 \leq |\phi_{n, xx}^k(L)| \leq b n^2
\ee
holds  and for all nontrivial components  $\phi_{n}^k$,  $k\in \{1, \ldots, N\}$.
\end{lemma}

\begin{proof}
As in the proof of Lemma \ref{asymptbehavior} we distinguish two cases as follows.\\
\noindent \textit{Case 1.} This case corresponds to eigenfunctions $\phi_n$ which have the form $\phi_n=(\Psi_n, \ldots, \Psi_n)$ where $(\sigma_n, \Psi_n)$ are eigenpairs of problem \eqref{eign.1.tree}. Let $\{\alpha_{1, n}\}_{n\geq 0}$, $\{\beta_{1, n}\}_{n\geq 0}$ be the sequences defined in the proof of Lemma \ref{asymptbehavior} satisfying the compatibility condition \eqref{eigbau}. Then in view of the boundary conditions at $x=0$ in \eqref{eign.1.tree}, we get
 \be
\Psi_n(x)=C_1\cos(\beta_{1, n}x) +C_3 \cosh (\alpha_{1, n} x).
\ee
In view of \eqref{eigbau} observe that $\cos(\beta_{1, n}L)\neq 0$  and since $\Psi_n(L)=0$  we obtain
\be\label{egfunct1}
\Psi_n(x)=-\frac{C_3 \cosh (\alpha_{1, n}L)}{\cos (\beta_{1, n}L)}\cos(\beta_{1, n}x) +C_3 \cosh (\alpha_{1, n} x).
\ee
The constant $C_3$ in \eqref{egfunct1} is chosen such that $\|\phi_n\|_{L^2(\Gamma)}=1 = \sqrt{N} \|\Psi_n\|_{L^2(0, L)}$, i.e.
$$C_3=\frac{\cos(\beta_{1, n}L)}{\sqrt{N}\sqrt{\int_0^L \left|\cosh(\alpha_{1, n}L)\cos(\beta_{1, n} x)+ \cosh(\alpha_{1, n}x)\cos(\beta_{1, n}L)\right|^2 dx}}. $$

Due to \eqref{eigbau}  after expanding the square  within the integral of $C_3$ we notice that the cross term is vanishing, i.e.
\be\label{egfunct2}
\int_0^L \cos(\beta_{1, n}x)\cosh(\alpha_{1, n}x)dx=0.
\ee
Indeed, using integration by parts we have the formula
$$\int_0^L \cos (\beta x) e^{\alpha x} dx= \frac{1}{\alpha^2+\beta^2}\left(\alpha \cos(\beta L) e^{\alpha L}-\alpha + \beta \sin (\beta L) e^{\alpha L} \right), \quad \forall \alpha, \beta \neq 0.$$
Applying it two times for $(\alpha, \beta)=(\alpha_{1, n}, \beta_{1, n})$ and $(\alpha, \beta)=(-\alpha_{1, n}, \beta_{1, n})$, respectively,  we obtain the validity of \eqref{egfunct2} taking into account \eqref{eigbau}. Therefore, due to \eqref{egfunct2}, \eqref{egfunct1} and the election of $C_3$  we obtain
\be\label{egfunct3}
\Psi(x)=\frac{-\cosh(\alpha_{1, n}L)\cos (\beta_{1, n}x)+\cos(\beta_{1,n}L) \cosh(\alpha_{1, n}x)}{\sqrt{N}\sqrt{ \cosh^2(\alpha_{1, n}L) \int_0^L \cos^2(\beta_{1, n} x) dx +  \cos^2(\beta_{1, n}L)\int_0^L \cosh^2(\alpha_{1, n}x) dx}}.
\ee
In consequence we get
\be\label{egfunct4}
\Psi_{n, xx}(L)=\frac{(\alpha_{1,n}^2 +\beta_{1, n}^2)\cosh(\alpha_{1, n}L)\cos (\beta_{1, n}L)}{\sqrt{N}\sqrt{ \cosh^2(\alpha_{1, n}L) \int_0^L \cos^2(\beta_{1, n} x) dx +  \cos^2(\beta_{1, n}L)\int_0^L \cosh^2(\alpha_{1, n}x) dx}}.
\ee
Then, in view of the behavior of $\{\alpha_{1,n}\}_{n\geq 0}, \ \{\beta_{1,n}\}_{n\geq 0}$ in \eqref{formula.betan} we successively have
\begin{align}\label{limit}
&\lim_{n\rightarrow \infty} \frac{\cosh(\alpha_{1, n}L)}{\sqrt{ \cosh^2(\alpha_{1, n}L) \int_0^L \cos^2(\beta_{1, n} x) dx +  \cos^2(\beta_{1, n}L)\int_0^L \cosh^2(\alpha_{1, n}x) dx}}\nonumber\\
&=\lim_{n\rightarrow \infty} \frac{\cosh(\alpha_{1, n}L)}{\sqrt{ \cosh^2(\alpha_{1, n}L) \Big(\frac{L}{2}+\frac{\sin(2\beta_{1, n}L)}{4\beta_{1, n}}\Big) +  \cos^2(\beta_{1, n}L)\Big(\frac{L}{2}+\frac{\sinh(2\alpha_{1, n}L)}{4\alpha_{1, n}}\Big)}}\nonumber\\
&=\sqrt{\frac{2}{L}}.
\end{align}
On the other hand, since $\cos (\beta_{1, n}L)\neq 0$ by \eqref{formula.betan} we also have
\be\label{infpos}
\inf_{n\geq 0} |\cos (\beta_{1, n}L)|>0.
\ee
Finally, combining \eqref{egfunct4}-\eqref{infpos} and the behavior of the sequences $\alpha_{1, n}, \beta_{1, n}$ (as $n\rightarrow \infty$) we obtain the behavior of $|\Psi_{n, xx}(L)|$ which give us the desired one for all components $(\phi^k_{n,xx}(L))_{k=1}^N$.\\

\noindent \textit{Case 2. } This case corresponds to eigenfunctions of the form $\phi_n=(c_1\Phi_n, c_2\Phi_n, \ldots, c_N \Phi_n )$, for  constants $\{c_k\}_{k=1}^N$ with $\sum_{k=1}^Nc_k=0$,  where $(\sigma_n, \Phi_n)$ are eigenpairs of the eigenvalue problem \eqref{eign2}. Let $\{\alpha_{2, n}\}_{n\geq 0}$, $\{\beta_{2, n}\}_{n\geq 0}$ be the sequences built in the proof of Lemma \ref{asymptbehavior} satisfying the compatibility condition \eqref{cucubau1} to ensure $\Phi_n\not\equiv 0$. Imposing the boundary conditions at $x=0$ in \eqref{eign2}  we get
\be\label{cucubau3}
\Phi_n(x)=C_2 \sin (\beta_{2, n}x) +C_4 \sinh(\alpha_{2, n} x).
\ee
Taking into account the normalization of the $L^2$-norm for $\Phi_n$, using similar steps as in Case 1 we finally obtain
\be
\Phi_n(x)=c_0\frac{-\sinh(\alpha_{2, n}L)\sin(\beta_{2, n}x)+\sin(\beta_{2,n}L) \sinh(\alpha_{2, n}x)}{\sqrt{ \sinh^2(\alpha_{2, n}L) \int_0^L \sin^2(\beta_{2, n} x) dx +  \sin^2(\beta_{1, n}L)\int_0^L \sinh^2(\alpha_{2, n}x) dx}},
\ee
where $c_0:=1/\sqrt{c_1^2+c_2^2+\ldots+c_N^2}$. Consequently
\be
\Phi_{n, xx}(L)=c_0\frac{(\alpha_{2, n}^2+\beta_{2, n}^2)\sinh(\alpha_{2, n}L)\sin(\beta_{2, n}L)}{\sqrt{ \sinh^2(\alpha_{2, n}L) \int_0^L \sin^2(\beta_{2, n} x) dx +  \sin^2(\beta_{1, n}L)\int_0^L \sinh^2(\alpha_{2, n}x) dx}}.
\ee
Similarly as in Case 1, one can show that
\be
\lim_{n\rightarrow \infty} \frac{\sinh(\alpha_{2, n}L)}{\sqrt{ \sinh^2(\alpha_{2, n}L) \int_0^L \sin^2(\beta_{2, n} x) dx +  \sin^2(\beta_{1, n}L)\int_0^L \sinh^2(\alpha_{2, n}x) dx}}=\sqrt{\frac{2}{L}}
\ee
and
\be
\inf_{n \geq 0} |\sin(\beta_{2, n}L)|>0,
\ee
which concludes the proof of the lemma.
\end{proof}

\subsection{Well-posedness}

In order to study the well-posedness of problem \eqref{control}-(I) we apply the semigroup theory.  First, let us consider the polynomial
\be\label{pol}
P(x)=\l(\frac{x}{L}\r)^4(x-L)
\ee
and denote $z^k(t,x)=y^k(t,x)-P(x)u^k(t)$, $k\in\{1, \ldots, N\}$. For our purpose it is more convenient to analyze first the equation satisfied by $z=(z^k)_{k=1, N}$. Indeed, it is easy to see that  $z$ satisfies the following nonhomogeneous problem with zero boundary conditions:
\be\label{wellposedness}
\left\{\begin{array}{ll}
z^k_t+ \la z^k_{xx}+z^k_{xxxx}=-Pu_t^k(t) -(\la P_{xx}+P_{xxxx})u^k(t), & (t, x)\in (0, T)\times (0,L)\\[10pt]
z^i(t, 0)=z^j(t, 0), & t\in (0, T), \quad  i,j \in \{1, \ldots, N\}\\[10pt]
z^k(t, L)= z_x^k(t, L)=0, & k\in \{1, \ldots, N\}\\[10pt]
\sum_{k=1}^{N}z^k_x(t, 0)=0, & t\in (0, T)\\[10pt]
z^i_{xx}(t, 0)=z^j_{xx}(t, 0), & t\in (0, T), \quad  i, j \in \{1, \ldots, N\}\\[10pt]
\sum_{k=1}^{N}z^k_{xxx}(t, 0)=0, & t\in (0, T)\\[10pt]
z^k(0, x)=y^k_0(x)-P(x)u^k(0), & x\in (0, L).
\end{array}\right.
\ee
Problem \eqref{wellposedness} can be written as an abstract Cauchy problem. Indeed, it follows that
 \be\label{abseqn}
\l\{\begin{array}{ll}
   z_t+A z= F(t,x, u), & t\in (0, T) \\
   z(0)=z_0, &
 \end{array}\right.
 \ee
where $A$ is the operator defined in \eqref{op}, whereas $F=(F^k)_{k=1,N}$ and $z_0=(z_0^k)_{k=1, N}$ are given by
$$F^k(t, x, u):=-P(x) u_t^k(t)-(\la P_{xx}(x)+P_{xxxx}(x))u^k(t),$$
respectively,
 $$z_0^k(x):=y^k_0(x)-P(x)u^k(0),$$ for any $k\in \{1, \ldots, N\}$.
In the previous section we have proved that $(A, D(A))$ generates a semigroup in $L^2(\Gamma)$. Therefore, applying the Hille-Yosida theory for the Cauchy problem \eqref{abseqn} (see, e.g., \cite[Proposition 4.1.6 and Lemma 4.1.5]{MR1691574}) we finally obtain
\begin{proposition}\label{wellpos}
For any $z_0 \in L^2(\Gamma)$ and $F\in L^1((0, T), L^2(\Gamma))$ there exists a mild function $z\in C([0, T], L^2(\Gamma))$ solution to \eqref{wellposedness}.
Moreover, if $z_0\in D(A)$ and $F\in C([0, T], L^2(\Gamma))\cap L^1((0, T), D(A))$ then the above mild solution satisfies
$z\in C([0, T], D(A))\cap C^1([0, T], L^2(\Gamma))$
solution to \eqref{wellposedness}. 
 \end{proposition}
 In our case the regularity of $F$ is ensured by the regularity of the controls we construct, namely $u\in H^{s}(0,T)$ for any $s>0$.
  Proposition \ref{wellpos} can be extended for   solutions of \eqref{control} with initial data $y_0\in D(A)$ or $y_0\in L^2(\Gamma)$.

\subsection{Controllability problem}\label{contPb}
Next we address the controllability problem \eqref{control}-(I) by using the method of moments \cite{MR0335014}. In view of that, let us consider the so-called adjoint problem, that is \be\label{adjoint}
\left\{\begin{array}{ll}
-q^k_t+ \la q^k_{xx}+q^k_{xxxx}=0, & (t, x)\in (0, T)\times (0,L),\\[10pt]
q^i(t, 0)=q^j(t, 0), & t\in (0, T), \quad  i,j \in \{1, \ldots, N\},\\[10pt]
q^k(t, L)= q_x^k(t, L)=0, & t\in (0, T), \q k\in \{1, \ldots, N\},\\[10pt]
\sum_{k=1}^{N}q^k_x(t, 0)=0, & t\in (0, T),\\[10pt]
q^i_{xx}(t, 0)=q^j_{xx}(t, 0), & t\in (0, T), \quad  i, j \in \{1, \ldots, N\},\\[10pt]
\sum_{k=1}^{N}q^k_{xxx}(t, 0)=0, & t\in (0, T),\\[10pt]
q^k(T, x)=q^k_T(x), & x\in (0, L).\\[10pt]
\end{array}\right.
\ee
Then, we have the following characterization of the null-controllability property.
\begin{lemma}\label{Lem_equiv}
System \eqref{control}-(I) is null-controllable in time $T>0$ if and only if, for any initial data $y_0=(y^k_0)_{k=1, N}\in L^2(\Gamma)$  there exists a control function $u=(u^k)_{k=1, N}\in (H^1(0,T))^N$    such that, for any  $q_T=(q^k_T)_{k=1, N}\in L^2(\Gamma)$
\be\label{nscond}
(y_0, q(0))_{L^2(\Gamma)}=\sum_{k=1}^{N}\intt u^k(t) q^k_{xx}(t, L) \dt,
\ee
where $q$ is the solution of \eqref{adjoint}.
\end{lemma}
\proof
We proceed as in the classical duality approach. We first multiply the equation in \eqref{control} by $q=(q^k)_{k=1, N}$, the solution of \eqref{adjoint}  to obtain
\[
\sum_{k=1}^{L}\intdoi \l(y^k_t+ \la y^k_{xx}+y^k_{xxxx}\r) q^k \dx \dt=0.
\]
 Integration by parts  leads to
\begin{align}\label{ident}
0=& \sum_{k=1}^{N} \intx y^k q^k \bdt \dx + \sum_{k=1}^{N} \intdoi \l(-q^k_t +\la q^k_{xx}+ q^k_{xxxx} \r) y^k \dx \dt \nonumber\\
&+ \intt \Big(\la y^k_x q^k \bdx - \la y^k q^k_x \bdx + y^k_{xxx} q^k \bdx  \nonumber\\
&\qquad\qquad \qquad -y^k_{xx} q^k_{x} \bdx + y^k_{x} q^k_{xx}\bdx - y^k q^k_{xxx}\bdx  \Big) \dx \dt.
\end{align}
In view of the boundary conditions satisfied by $y=(y^k)_{k=1, N}$ and $q=(q^k)_{k=1, N}$, identity  \eqref{ident} is equivalent to
\be\label{identcontrol}
\sum_{i=1}^{N} \intx (y^k(T, x)q^k(T, x)-y^k_0(x)q^k_0(x)) \dx + \intt u^k(t) q^k_{xx} (t, L)\dt=0.
\ee

\textit{``Only if" implication.}   Since \eqref{control}-(I) is null-controllable (i.e. $y(T, x)=0$ for any $x\in \Gamma $) it follows from \eqref{identcontrol} that condition \eqref{nscond} holds true.

\textit{``If" implication.} Let us assume the validity of \eqref{nscond}. In this case, due to \eqref{identcontrol} we get
$$(y(T), q_T)_{L^2(\Gamma)}=0,$$
for any $q_T \in L^2(\Gamma)$. This implies $y(T)=0$.
\endproof
As a consequence of Lemma \ref{Lem_equiv} and the spectral analysis developed in subsection \ref{spectanal} the controllability problem reduces to the following moment problem.
\begin{lemma}\label{momentpb}
Let $\{\sigma_n\}_{n \geq 0}$ be the distinct eigenvalues of system \eqref{spect} and denote by $m(\sigma_n)$  the multiplicity of each eigenvalue $\sigma_n$  whose eigenspace is generated by the linear independent eigenvectors  $\{\phi_{n, l}\}_{l=1, m(\sigma_n)}$, normalized in $L^2(\Gamma)$.     System \eqref{control}-(I) is null-controllable in time $T>0$ if and only if for any initial data $y_0\in L^2(\Gamma)$,
\be\label{s2}
y_0=\sum_{n\geq 0}\sum_{l=1}^{m(\sigma_n)} y_{0, n, l} \phi_{n, l}\textcolor{red}{,}
\ee 
there exists a control $u=(u^k)_{k=1, N}\in (H^1(0,T))^N$ such that
\be\label{ort}
y_{0, n, l}e^{_-T \sigma_n} =\sum_{k=1}^{N} \phi^k_{n, l, xx}(L) \intt u^k(T-t) e^{-t \sigma_n} \dt, \quad \forall n\geq 0,  \forall l=1,\dots, m(\sigma_n).
\ee
\end{lemma}
\proof
For any $q_T\in L^2(\Gamma)$ we write
$$q_T =\sum_{n\geq 0}  \sum_{l=1}^{m(\sigma_n)} q_{n, l} \phi_{n, l}, $$
where $\sum_{n\geq 0} \sum_{l=1}^{m(\sigma_n)}|q_{n, l}|^2 < \infty$.
Then, seeking for solutions of system \eqref{adjoint} in separable variable
$$q(t, x)=\sum_{n\geq 0} \sum_{l=1}^{m(\sigma_n)} \ov{q}_{n, l}(t) \phi_{n, l} (x),$$
the time coefficients $\ov{q}_{n, l}$ satisfy
$$\ov{q}^{'}_{n, l}-\sigma_n \ov{q}_{n, l}=0, \quad \ov{q}_{n, l}(T)=q_{n, l}.$$
Then, we obtain
$$q(t, x)=\sum_{n\geq 0}\sum_{l=1}^{m(\sigma_n)}e^{(-T+t)\sigma_n}q_{n, l} \phi_{n, l}(x),$$
and therefore
\be\label{s1}
q_{xx}(t, L)=\sum_{n\geq 0} \sum_{l=1}^{m(\sigma_n)} e^{(-T+t)\sigma_n}q_{n, l} \phi_{n, l, xx}(L).
\ee
Plugging \eqref{s1} and \eqref{s2} in the controllability condition \eqref{nscond} we obtain that the existence of a function $u$ satisfying the moment problem \eqref{ort} is equivalent to  the null-controllability property.
\endproof

\subsection{Proof of Theorem \ref{main theorem 1}}\label{prooft1}
We show that a control acting on $N-1$ nodes is sufficient to obtain the null-controllability of system \eqref{control}-(I).

According to Lemma \ref{momentpb} we have  to solve the moment problem \eqref{ort} by constructing a control $u\in (H^1(0,T))^N$. In order to do that we settle one of the components of $u=(u^k)_{k=1, N}$ to be identically zero.  For simplicity, without loss of generality we assume $u^N\equiv 0$. Then the problem of moments \eqref{ort} becomes
\be\label{cond}
y_{0, n, l}e^{_-T \sigma_n} =\sum_{k=1}^{N-1} \phi^k_{n, l, xx}(L) \intt u^k(T-t) e^{-t \sigma_n} \dt, \quad \forall n\in \nn, \ \forall l=1, \dots, m(\sigma_n).
\ee

  In fact, in view of the analysis done in the proof of Lemma \ref{asymptbehavior} we  have $m(\sigma_n)\in \{1, N-1\}$. More precisely, if $\sigma_n$ is simple eigenvalue (i.e. $m(\sigma_n)=1$) then, in view of Lemma \ref{eigen_partition}  we  choose $\phi_{n, 1}=(\Psi_n, \ldots, \Psi_n)$ where $(\sigma_n, \Psi_n)$ is an eigenpair of problem \eqref{eign.1.tree} such that $\|\phi_{n, 1}\|_{L^2(\Gamma)}=1$. On the other hand, if $m(\sigma_n)=N-1$ we  choose for any $l\in \{1, \ldots, N-1\}$
\begin{equation}\label{multiple}
  \phi_{n, l}=\Phi_n e_l -\Phi_n e_{l+1},
\end{equation}
where $\{e_i\}_{i=1, N}$ is the canonical basis of the euclidian space $\rr^N$ and $(\sigma_n, \Phi_n)$ is an eigenpair of problem \eqref{eign2} such that
$\|\phi_{n,l}\|_{L^2(\Gamma)}=1$. Then $\{\phi_{n, l}\}_{l=1, N-1}$ form an orthonormal basis of the eigenspace of $\sigma_n$.
If $m(\sigma_n)=1$ then \eqref{cond} becomes
\be\label{cond1}
y_{0, n, 1}e^{_-T \sigma_n} =\Psi_{n, xx}(L)\sum_{k=1}^{N-1}  \intt u^k(T-t) e^{-t \sigma_n} \dt, \quad \forall n\geq 0.
\ee
Therefore, it suffices that the control inputs $u^k\in L^2(0, T)$, $k\in \{1, \ldots, N-1 \}$  satisfy the following moment problems
\be\label{defcontrol}
\left\{\begin{array}{ll}
  \int_0^T u^1(T-t) e^{-t\sigma_n} dt=\frac{y_{0, n, 1}e^{-T \sigma_n}}{\Psi_{n, xx}(L)}, &   \forall n\geq 0 \\[10pt]
  \int_0^T u^k(T-t) e^{-t\sigma_n} dt=0, &\textrm{ if }  k=2, \dots, N-1, \quad \forall n\geq 0,\\
\end{array}\right.
\ee
where, by convention, the second line in  \eqref{defcontrol} is not taken into account when $N=2$.
On the other hand, if $\sigma_n$ is a multiple eigenvalue (i.e. $m(\sigma_n)=N-1$) then \eqref{cond} becomes
\be\label{cond3}\left\{\begin{array}{lll}
y_{0, n, l}e^{-T \sigma_n} = \Phi_{n, xx}(L) \left(\intt u^l(T-t) e^{-t \sigma_n} \dt-\intt u^{l+1}(T-t) e^{-t \sigma_n} \dt\right) & \\[3pt]
 \hspace{3cm}  \forall n\geq 0, \   \forall l=1,\dots,  N-2 &\\[5pt]
y_{0, n, N-1}e^{-T \sigma_n} = \Phi_{n, xx}(L) \intt u^{N-1}(T-t) e^{-t \sigma_n}, \   \forall n\geq 0,
\end{array}\right.
\ee
where we make the convection that when $N=2$ the first relation in \eqref{cond3} does not appear.
Equivalently we can write
\begin{equation}\label{syst}
\left(
  \begin{array}{c}
    y_{0, n, 1} \\
    \vdots \\
    \vdots \\
      y_{0, n, N-1} \\
  \end{array}
\right) e^{-T \sigma_n}= \Phi_{n, xx}(L) M
 \left(
  \begin{array}{c}
    \int_0^T u^1(T-t)e^{-t \sigma_n} \dt  \\
    \vdots \\
    \vdots \\
      \int_0^T u^{N-1}(T-t)e^{-t \sigma_n} \dt \\
  \end{array}
\right)
\end{equation}
where matrix $M$ is given by 
$$M=(a_{ij})_{i,j=1, N-1}:= \left(
 \begin{array}{cccccc}                                             1& -1 & 0 & \ldots & \ldots  & 0   \\
 0& 1 & -1 & 0 & \ldots & 0\\                                            \vdots &  \ldots & \ldots & \ddots & \ddots  & 0 \\
  0& \ldots & \ldots & 0 & 1 & -1 \\                                            0& \ldots & \ldots & \ldots & 0 & 1 \\
 \end{array}\right).$$
 Let us denote by $M^{-1}=(a^{ij})_{i,j=1, N-1}$ the inverse of matrix $M$.
Then we get
\begin{equation}\label{syst2}
  \left(
  \begin{array}{c}
    \int_0^T u^1(T-t)e^{-t \sigma_n} \dt  \\
    \vdots \\
    \vdots \\
      \int_0^T u^{N-1}(T-t)e^{-t \sigma_n} \dt \\
  \end{array}
\right) =\frac{1}{\Phi_{n, xx}(L)} M^{-1} \left(
  \begin{array}{c}
    y_{0, n, 1} \\
    \vdots \\
    \vdots \\
      y_{0, n, N-1} \\
  \end{array}
\right) e^{-T\sigma_n}.
 \end{equation}
To summarize, the moment problems to be solved are
\begin{equation}\label{momprob}
\int_0^T u^k(T-t)e^{-t \sigma_n} \dt=c_{n, k}, \quad n \geq 0,
\end{equation}
for each $k=1, \dots, N-1$, where
$c_{n,k}$ are given by
\begin{equation*}
  c_{n, 1}:=\left\{\begin{array}{ll}
              \frac{y_{0, n, 1}e^{-T\sigma_n}}{\Psi_{n,xx}(L)}, & \textrm{if } m(\sigma_n)=1, \\[3pt]
              \sum_{j=1}^{N-1} \frac{a^{1j}y_{0, n,j}e^{-T\sigma_n}}{\Phi_{n, xx}(L)},& \textrm{if } m(\sigma_n)=N-1,
            \end{array}\right.
\end{equation*}
respectively, for any $k\geq 2$ (only when $N\geq 3$)
\begin{equation*}
 \quad   c_{n, k}:=\left\{\begin{array}{ll}
              0, & \textrm{if } m(\sigma_n)=1, \\[3pt]
              \sum_{j=1}^{N-1} \frac{a^{kj}y_{0,n, j}e^{-T\sigma_n}}{\Phi_{n, xx}(L)}, & \textrm{if } m(\sigma_n)=N-1.
            \end{array}\right.
\end{equation*}

In view  of \cite[(3.9)]{MR0335014},
to solve the moment problem  \eqref{momprob}, it is enough to show that the series $$\sum_{n\geq n_0} c_{n, k} \sigma _n^{1/2} \prod_{j=1, j\neq n}^{\infty}\frac{\sigma_n+\sigma_j}{\sigma_n-\sigma_j}, \quad \forall k=1, \dots, N-1,$$
is absolutely convergent. Indeed this is a consequence of 
 the asymptotic behavior of the positive eigenvalues $\{\sigma_n\}_{n\geq n_0}$ given in  Lemma \ref{asymptbehavior} and the asymptotic properties of the corresponding eigenfunctions $\{\phi_n\}_{n\geq n_0}$ in Lemma \ref{asy_eigenfunct}. 
%
%
 Thus we are able to build the controls $(u^k)_{k=1, N-1}$ verifying  \eqref{momprob} and finalize the proof.\\

\noindent \textbf{Optimality of $N-1$ controls}.  For the proof of the second statement in Theorem \ref{main theorem 1} let us consider without loss the generality that $i_0=1$ and $j_0=2$ and  assume that we can control system \eqref{control}-(I) with $N-2$ controls, i.e. $u^1\equiv u^2\equiv 0$.
We consider as the initial datum $y_0=(\Phi, -\Phi, 0, \ldots, 0)$, where $\Phi$ is an  eigenfunction of system \eqref{eign2} corresponding to some eigenvalue $\sigma$. Then system \eqref{ort} to be solved 
\[
e^{-T \sigma} =\phi_{xx}(L) \intt (u^1(T-t)-u^2(T-t)) e^{-t \sigma} \dt
\] 
has no solution since the right hand side above is vanishing.
Therefore, such $y_0$ cannot be driven to zero by any control $u=(u^k)_{k=1,N}$ with the first two components vanishing. The proof of the theorem is now complete.
\endproof

\section{The KS equation of type (II)}\label{secII}

The results obtained  in this section are based on a careful analysis of the eigenvalues for the corresponding elliptic operator of system \eqref{control}-(II).  In the previous section we addressed this problem for system \eqref{control}-(I) by using specific spectral results obtained in \cite{MR2556747}.
 In the present case such results are not applicable. Therefore, to take advantage of the strategy implemented in the previous case  an additional work has to be done by determining explicitly the spectrum and the eigenfunctions for two different eigenvalue problems as follows.
\subsection{Preliminaries I}\label{prelim1}
In this subsection  we analyze the following eigenvalue problem
\be\label{eigenval}
\left\{\begin{array}{ll}
 \la \phi_{xx}+\phi_{xxxx}=\sigma \phi, & x\in (0,L)\\[10pt]
\phi_x(0)=\phi_{xxx}(0)=0, & \\[10pt]
\phi_{x}(L)=\phi_{xxx}(L)=0. & \\[10pt]
\end{array}\right.
\ee
Again, as in section \ref{spectanal} we can employ spectral analysis tools to show that system \eqref{eigenval} has a sequence of eigenvalues which tends to infinity and is bounded from below by $-\lambda^2/4$.

Making use of the characteristic equation of the equation in \eqref{eigenval}, i.e.,
\[
r^4+\lambda r^2 -\sigma =0
\]
in view of the notations in \eqref{notations} we distinguish several cases as follows.
\subsection*{Case I: $\sigma>0$} In this case, the general solution of the equation in \eqref{eigenval} is given by
$$\phi(x)=C_1 \cosh (\alpha x)+C_2 \sinh (\alpha x)+C_3 \cos (\beta x)+C_4 \sin (\beta x),$$
where $C_i$ are real constants, $i=1, \dots, 4$.
Imposing the boundary conditions at $x=0$ we easily obtain that $C_2=C_4=0$. The boundary conditions at $x=L$  provide a nontrivial solution $\phi$  if  $\sinh (\alpha L) \sin (\beta L)=0$. Since $\alpha>0$ this is equivalent to the compatibility condition
\[
\sin (\beta L)=0.
\]
Then, we get a sequence $\{\beta_n\}_{n\geq 1}$ of positive solutions, $\beta_n=n\pi/L$. In view of \eqref{notations} we obtain that the sequence of positive simple eigenvalues is given by
\[
\sigma_n =\left(\frac{n\pi}{L}\right)^4-\lambda \left(\frac{n \pi}{L}\right)^2, \quad n>  \frac{L\sqrt{\lambda}}{\pi}.
\]
The corresponding eigenfunctions are
\[
\phi_n(x)=C_1 \cos(\beta_n x), \quad C_1\neq 0.
\]

\subsection*{Case II: $\sigma=0$} The general solution of the equation in \eqref{eigenval} is
$$ \phi(x)=C_1+C_2 x +C_3 \cos (\sqrt{\lambda}x)+ C_4 \sin (\sqrt{\lambda}x)\textcolor{red}{.}$$
From the boundary conditions at $x=0$ we deduce  that $C_2=C_4=0$. According to the boundary conditions at $x=L$ we have the following alternatives:
\begin{enumerate}
\item If $\sin(\sqrt{\lambda} L)=0$, i.e. $\lambda\in \mathcal{N}_2$, then $\sigma=0$ is an eigenvalue with multiplicity two and the eigenfunctions are
$$\phi_0(x)=C_1+C_3 \cos (\sqrt{\lambda}x), \quad C_1^2+C_3^2\neq 0.$$
\item  If $\sin(\sqrt{\lambda} L) \neq 0$, $\lambda\notin \mathcal{N}_2$, then $\sigma=0$ is a simple eigenvalue and  the eigenfunctions are constant functions, {i.e.,}
$$\phi_0(x)=C, \quad C\neq 0.$$
\end{enumerate}

\subsection*{Case III: $-\lambda^2/4< \sigma<0$} The general solution of the equation in \eqref{eigenval} is
$$\phi(x)=C_1 \cos (\gamma x)+C_2 \sin (\gamma x)+C_3 \cos (\beta x)+C_4 \sin (\beta x),$$
where $\beta\neq \gamma$.
The boundary conditions at $x=0$ lead to $C_2=C_4=0$. Defining the quantity
$\delta:=\sin (\beta L)\sin (\gamma L)$, from the boundary conditions at $x=L$ we obtain that $\phi$ is an eigenfunction if and only if $\delta=0$. From $\delta=0$ we deduce that that all the eigenvalues are
\[
\sigma_n = \left(\frac{n\pi}{L}\right)^4-\lambda \left(\frac{n \pi}{L}\right)^2, \quad 1\leq n <   \frac{L\sqrt{\lambda}}{\pi},
\]

We distinguish the following cases for $\delta=0$.

\begin{enumerate}
\item \textit{The case $\sin (\beta L)=\sin (\gamma L)=0$}, i.e. $\lambda\in \mathcal{N}_1$. We obtain two sequences of solutions, $\beta_n=n\pi/L$ and $\gamma_m=m\pi/L$ with $n\neq m$ (since $\beta\neq \gamma$), $n, m\geq 1$.   From \eqref{connection} this is equivalent to 
$$\lambda = (\frac{n\pi}L)^2+(\frac{m\pi}L)^2\in \mathcal{N}_1$$ and in this case 
$\sigma_n=\sigma_m$ and has multiplicity two 
%
%
with  the corresponding eigenfunctions 
  \[
  \phi_{n, m}(x)= C_1  \cos \left(\frac{n \pi x}{L}\right)+C_3  \cos \left(\frac{m \pi x}{L}\right),  \quad C_1^2+C_3^2 \neq 0.
  \]
	  \item  \textit{ The case $\delta=0,$ such that $\sin^2(\beta L)+\sin^2(\gamma L)> 0$, i.e. $\lambda\notin \mathcal{N}_1$}. 
In this case all the eigenvalues $\sigma_n$ are simple and  the corresponding eigenfunctions are
\[
\phi_n(x)=C_3 \cos \left(\frac{n \pi x}{L}\right), \quad C_3\neq 0.
\]
\end{enumerate}
From the spectral analysis developed above it is easy to check the following lemma that will play an important role in the proof of Theorem \ref{main theorem 2}.

\begin{lemma}\label{lem1} Let $\lambda>0$ and $(\sigma, \phi)$ be an eigenpair of system \eqref{eigenval}. The following holds:
\begin{enumerate}

\item If $\sigma> 0$ then
\be
\phi(L)\neq 0 \textrm{  and } \lambda \phi(L)+\phi_{xx}(L)\neq 0.\nonumber
\ee
\item If $\sigma=0$ and $\lambda\in \mathcal{N}_2$ then $\phi(L)$ and $\lambda \phi(L)+\phi_{xx}(L)$ cannot vanish simultaneously.
\item If $\sigma=0$ and $\lambda\not \in \mathcal{N}_2$ then $\phi(L)\neq 0$ and $\lambda \phi(L)+\phi_{xx}(L)\neq 0$.
\item  If $\sigma<0$ and $\lambda\in \mathcal{N}_1$ then $\phi(L)$ and $\lambda \phi(L)+\phi_{xx}(L)$ cannot vanish simultaneously.
 \item If $\sigma <0$ and $\lambda \not \in \mathcal{N}_1$ then
\be
\phi(L)\neq 0 \textrm{ and } \lambda \phi(L)+\phi_{xx}(L)\neq 0.\nonumber
\ee
\end{enumerate}
\end{lemma}
\subsection{Preliminaries II}\label{prelim2}
Secondly  we analyze the following eigenvalue problem
\be\label{eigenval2}
\left\{\begin{array}{ll}
 \la \phi_{xx}+\phi_{xxxx}=\sigma \phi, & x\in (0,L),\\[10pt]
\phi(0)=\phi_{xx}(0)=0, & \\[10pt]
\phi_{x}(L)=\phi_{xxx}(L)=0. & \\[10pt]
\end{array}\right.
\ee
Again, the spectrum of \eqref{eigenval2} is pure discrete, strictly bounded from below by $-\lambda^2/4$ and tends to infinity. Making use of the notations \eqref{notations}     we distinguish the following cases.
\subsection*{Case I: $\sigma>0$}
The general solution of the equation in \eqref{eigenval2} is
$$\phi(x)=C_1 \cosh (\alpha x)+C_2 \sinh (\alpha x)+C_3 \cos (\beta x)+C_4 \sin (\beta x).$$
From the boundary conditions at $x=0$ we easily obtain that $C_1=C_3=0$. The conditions at $x=L$  say that $\phi$ is an  eigenfunction under the constraint
\[
\cos (\beta L)=0.
\]
We get a sequence $\{\beta_n\}_{n\geq 0}$ of positive solutions,  $\beta_n=(2n+1)\pi/2L$. In view of \eqref{notations}-\eqref{connection} we obtain the sequence of simple eigenvalues $\sigma_n =\beta_n^2(\beta_n^2 -\lambda)$, i.e.
\[
\sigma_n = \Big(\frac{(2n+1)\pi}{2L}\Big)^4 -\lambda \Big(\frac{(2n+1)\pi}{2L}\Big)^2, \quad 2n+1>\frac {2L\sqrt \lambda}\pi,
\]
with the corresponding family of eigenfunctions
$$\phi_n (x)=C_4 \sin (\beta_n x), \quad C_4\neq 0.$$
\subsection*{Case II: $\sigma=0$} The general solution for the equation in \eqref{eigenval2} is
$$ \phi(x)=C_1+C_2 x +C_3 \cos (\sqrt{\lambda}x)+ C_4 \sin (\sqrt{\lambda}x).$$
From the boundary conditions at $x=0$ we deduce that $C_1=C_3=0$. Then, the boundary conditions at $x=L$ produce the following cases:
 \begin{enumerate}
 \item If $\cos(\sqrt{\lambda} L)=0$, which is equivalent to $\lambda  \in \mathcal{N}_3$,  then $\sigma=0$ is a simple eigenvalue with the corresponding eigenfunctions
$$\phi_0(x)=C_4 \sin (\sqrt{\lambda} x), \quad C_4\neq 0.$$
 \item If $\cos(\sqrt{\lambda} L)\neq 0$, i.e. $\lambda \not \in \mathcal{N}_3$, then $\sigma=0$ is not an eigenvalue.
\end{enumerate}
\subsection*{Case III: $-\lambda^2/4< \sigma<0$} The general solution of the equation in \eqref{eigenval2} is
$$\phi(x)=C_1 \cos (\gamma x)+C_2 \sin (\gamma x)+C_3 \cos (\beta x)+C_4 \sin (\beta x).$$
The boundary conditions at $x=0$ give $C_1=C_3=0$.
Then, from the conditions at $x=L$ we obtain that $\phi$ is an eigenfunction if and only if
$
\delta:=\cos (\beta L)\cos (\gamma L)=0.
$
We deduce that the eigenvalues are 
\[
\sigma_n = \Big(\frac{(2n+1)\pi}{2L}\Big)^4 -\lambda \Big(\frac{(2n+1)\pi}{2L}\Big)^2, \quad 2n+1<\frac {2L\sqrt \lambda}\pi.
\]
 We distinguish the following cases for $\delta=0$:
 \begin{enumerate}
  \item \textit{The case $\cos (\beta L)=\cos (\gamma L)=0$, i.e. $\lambda \in \mathcal{N}_{odd}$}. We obtain $\beta_n=(2n+1)\pi/2L$ and $\gamma_m=(2m+1)\pi/2L$, with $n\neq m$ (since $\beta\neq \gamma$) and 
   \[
 \lambda = (\frac{(2n+1)\pi}{2L})^2+(\frac{(2m+1)\pi}{2L})^2\in \mathcal{N}_{odd}
  \]
In this case the eigenvalues $\sigma_n=\sigma_m$ have multiplicity two and 
  the corresponding eigenfunctions are
  \[
  \phi_{n, m}(x)= C_2 \sin(\beta_n x)+C_4 \sin (\beta_m x),  \quad C_2^2+C_4^2 \neq 0.
  \]
 \item \textit{The case $\cos^2 (\beta L)+\cos^2 (\gamma L)>0$, i.e. $\lambda \not \in \mathcal{N}_{odd}$}. 
 In this case all the eigenvalues $\sigma_n$ are simple and  
  the corresponding eigenfunctions are given by
\[
\phi_n(x)= C \sin \left(\frac{(2n+1)\pi x}{2L}\right), \quad C\neq 0.
\]
\end{enumerate}
Combining the spectral results of this section we conclude
\begin{lemma}\label{lem2} Let $\lambda>0$ and $(\sigma, \phi)$ be an eigenpair of system \eqref{eigenval2}.
\begin{enumerate}
\item If $\sigma> 0$ then
\be
\phi(L)\neq 0 \textrm{  and } \lambda \phi(L)+\phi_{xx}(L)\neq 0. \nonumber
\ee
\item If $\lambda \in \mathcal{N}_3$ then $\sigma=0$ is an eigenvalue and
\be
\phi(L)\neq 0,   \lambda \phi(L)+\phi_{xx}(L)=0.\nonumber
\ee
\item If $\lambda \not \in \mathcal{N}_3$ then $\sigma=0$ is not an eigenvalue.
\item If $\sigma <0$ and $\lambda  \in \mathcal{N}_{odd}$ then $\phi(L)$ and $\lambda \phi(L)+\phi_{xx}(L)$ cannot vanish simultaneously.
\item If $\sigma<0$ and $\lambda \not \in \mathcal{N}_{odd}$ then
\be
\phi(L)\neq  0 \textrm{  and } \lambda \phi(L)+\phi_{xx}(L)\neq 0. \nonumber
\ee
\end{enumerate}
\end{lemma}

\subsection{Spectral analysis}

In this section we aim to discuss some general properties of the following spectral problem
\be\label{spect2}
\left\{\begin{array}{ll}
 \la \phi^k_{xx}+\phi^k_{xxxx}=\sigma \phi^k, & x\in (0,L),\\[10pt]
\phi^i(0)=\phi^j(0), &  \quad  i,j \in \{1, \ldots, N\},\\[10pt]
\phi_x^k(L)=\phi_{xxx}^k(L)=0, & k\in \{1, \ldots, N\},\\[10pt]
\sum_{k=1}^{N}\phi^k_x(0)=0, & \\[10pt]
\phi^i_{xx}(0)=\phi^j_{xx}(0), & \quad  i, j \in \{1, \ldots, N\},\\[10pt]
\sum_{k=1}^{N}\phi^k_{xxx}(0)=0, & \\[10pt]
\end{array}\right.
\ee
which governs our control system \eqref{control}-(II).
This is equivalent to study the spectral properties of the fourth order operator
$$A:  D(A)\subset L^2(\Gamma)\rightarrow L^2(\Gamma)$$ given by
\be\label{op2}
\l\{\begin{array}{ll}
      A\phi^k=\la \phi^k_{xx}+\phi^k_{xxxx},\quad k \in \{1, \ldots, N\}  \\[10pt]
      D(A)=\l\{\begin{array}{ll} \phi=(\phi^k)_{k=1, N} \in H^4(\Gamma) \ | \ \phi_x^k(L)=\phi_{xxx}^k(L)=0, \quad k \in \{1, \ldots, N\}, \\[5pt] \phi^i(0)=\phi^j(0), \  \phi^i_{xx}(0)=\phi^j_{xx}(0), \quad i, j \in \{1, \ldots, N\},\\[5pt]
                 \sum_{k=1}^{N}\phi^k_x(0)=0, \          \sum_{k=1}^{N}\phi^k_{xxx}(0)=0,
               \end{array}\r\}.
    \end{array}\r.
\ee
Similar to the operator induced by the model \eqref{control}-(I) we obtain the following characterisation of $A+\mu I$.
 \begin{proposition}\label{p12}
 For any $\mu> \la^2/4$ the operator
 $$A+\mu I :D(A)\subset L^2(\Gamma)\rightarrow L^2(\Gamma),$$
 is a non-negative self-adjoint operator with compact inverse. In particular, it has a pure discrete spectrum formed by a sequence of nonnegative eigenvalues $\{\sigma_{\mu, n}\}_{n\geq 0}$ satisfying $\lim_{n\to \infty}\sigma_{\mu, n}=\infty$.  Moreover, up to a normalization, the corresponding eigenfunctions $\{\phi_{\mu, n}\}_{n\geq 0}$ form an orthonormal basis of $L^2(\Gamma)$.
 \end{proposition}
In consequence the spectrum of $A$ has the following properties.
 \begin{proposition}\label{p13}
The spectrum $\{\sigma_n\}_{n\geq 0}$ of problem \eqref{spect2} verifies
\begin{equation}\label{spectrumA2}
-\frac{\lambda^2}{4}<  \sigma_n, \ \forall n\geq 0, \quad \quad \sigma_n \rightarrow \infty, \textrm { as }n\rightarrow \infty.
\end{equation}
 \end{proposition}
The details of the proof of Propositions \ref{p12} and \ref{p13} are  similar as in  model \eqref{control}-(I), therefore they will be omitted here.
\subsection{Qualitative properties of the eigenvalues}\label{spectqual2}

 In this section we apply the preliminary results shown in subsections \ref{prelim1} and \ref{prelim2} to obtain useful spectral properties of system \eqref{spect2} which will play a crucial role for proving the null-controllability of problem \eqref{control}-(II).  

For any fixed $\lambda>0$  let us first consider the following two eigenvalue problems which have been already analyzed in subsections \ref{prelim1} and \ref{prelim2}:
\begin{equation}\label{eign.1.tree2}
E_1:\left\{
\begin{array}{ll}
\lambda \Psi_{xx}+\Psi_{xxxx}=\sigma \Psi, & x\in (0,L),\\[10pt]
 \Psi_x(0)=0, \ \Psi_{xxx}(0)=0 , &\\[10pt]
 \Psi_x(L)=\Psi_{xxx}(L)=0, &\\[10pt]
\end{array}
\right.
\end{equation}
and
\begin{equation}\label{eign22}
E_2:\left\{
\begin{array}{ll}
\lambda \Phi_{xx}+\Phi_{xxxx}=\sigma \Phi, & x\in (0,L),\\[10pt]
 \Phi(0)=0, \ \Phi_{xx}(0)=0 , &\\[10pt]
 \Phi_x(L)=\Phi_{xxx}(L)=0, &\\[10pt]
\end{array}
\right.
\end{equation}
respectively. Recall that Lemma \ref{lem1} applies to \eqref{eign.1.tree2} whereas Lemma \ref{lem2} applies to \eqref{eign22}.
Coming back to our spectral problem \eqref{spect2},  we introduce the functions
\begin{equation}\label{functie_suma2}
S:=\sum_{k=1}^{N} \phi^k,
\end{equation}
\begin{equation}\label{functie_diferenta2}
D^k:=\phi^k - \frac{S}{N},  \quad k\in \{1, \ldots, N\}.
\end{equation}
The motivation for analyzing systems \eqref{eign.1.tree2} and \eqref{eign22} is due to the fact that $S$ verifies \eqref{eign.1.tree2} whereas $D^k$ satisfies \eqref{eign22} for all $k\in \{1, \ldots, N\}$.

Next we state and prove some preliminary lemmas.

\begin{lemma}\label{not.common2}
	For any $\lambda>0$ the eigenvalue problems \eqref{eign.1.tree2} and \eqref{eign22} have no any common positive eigenvalue.
The  value $\sigma=0$ is a common eigenvalue if and only if  $\lambda$ belongs to $\mathcal{N}_3$.
   Moreover, problems \eqref{eign.1.tree2} and \eqref{eign22} have no  common negative eigenvalues if and only if $\lambda \not \in \mathcal{N}_4$.
\end{lemma}

\begin{proof}
Assume that $\sigma>0$ is a common eigenvalue for \eqref{eign.1.tree2} and \eqref{eign22}. Then, according to the precise analysis in subsections \ref{prelim1} and \ref{prelim2} we must necessary have
$$\sin(\beta L)=\cos (\beta L)=0, $$
 which never may happen.
 	Again, in view of the subsections above we obtain  that $\sigma=0$ is an eigenvalue of \eqref{eign.1.tree2} but it cannot be an eigenvalue of \eqref{eign22} unless $\lambda \in \mathcal{N}_3$.
 Moreover, if some $\sigma<0$ was a common eigenvalue we should have
 $$\sin(\beta L)\sin (\gamma L)=\cos (\beta L)\cos (\gamma L)=0,$$
 which is equivalent to the alternatives
 $$\sin(\beta L)=\cos(\gamma L)=0 \textrm{ or }\sin (\gamma L)=\cos (\beta L)=0.$$
 This is
  equivalent with $\lambda \in \mathcal{N}_{4}$.
\end{proof}

In consequence when $\lambda\notin\mathcal{N}_{mixt}=\mathcal{N}_3\cup \mathcal{N}_4$ we obtain that there is no common eigenvalue of systems  \eqref{eign.1.tree2} and \eqref{eign22}. This allows us to give a precise characterisation of 
the eigenvalues of system \eqref{spect2}.

\begin{lemma}\label{eigen_partition3}
Assume $\lambda\not \in \mathcal{N}_{mixt}$ then
\be\label{equiparti}
\sigma_p(A)=\sigma_p(E_1)\cup \sigma_p(E_2); \quad \sigma_p(E_1)\cap \sigma_p (E_2)=\emptyset,
\ee
where $\sigma_p(A)$, $\sigma_p(E_1)$ and $\sigma_p(E_2)$  denote the set of eigenvalues for the spectral problems \eqref{spect2}, \eqref{eign.1.tree2} and \eqref{eign22}, respectively.  In addition, we can  precisely describe the eigenpairs of \eqref{spect2} as follows.
 If $(\sigma, \phi=(\phi^k)_{k=1,N})$ is an eigenpair of \eqref{spect2} we have the following possibilities:
    \begin{enumerate}
    \item If $\sigma>0$ then we have the alternative:
    \begin{enumerate}
    \item\label{a} either $\sigma$ is an  eigenvalue of \eqref{eign.1.tree2} and there exists $\Psi$ an eigenfunction of $\sigma$ in \eqref{eign.1.tree2} such that
        $$\phi=(\Psi, \ldots, \Psi).$$
     In this case $\sigma$ is a simple eigenvalue of \eqref{spect2} (i.e. $m(\sigma)=1$) and a basis for its eigenpspace is given
     by
     $$\mathcal{B}_\sigma=\{(\Psi, \ldots \Psi)\}.$$

    \item\label{b} or $\sigma$ is an eigenvalue of \eqref{eign22}, situation in which  there exists an eigenpair   $(\sigma, \Phi)$ of \eqref{eign22} and   a nontrivial  vector  $\overrightarrow{C}=(c_k)_{k=1, N}$ with $\sum_{k=1}^{N}c_k=0$ such that
        $$\phi=\overrightarrow{C} \Phi.$$ In this case $\sigma$, as an eigenvalue of \eqref{spect2}, has multiplicity $m(\sigma)=N-1$. A  basis for the eigenspace of $\sigma$ is given by
        $$\mathcal{B}_{\sigma}=\{\Phi{\bf e}_{l}-\Phi{\bf e}_{l+1}\}_{l=1, N-1}.$$
  \end{enumerate}
  \item If $\sigma =0$ then $\sigma $ is an eigenvalue of \eqref{eign.1.tree2}. We distinguish two cases:
      \begin{itemize}
           \item If $\lambda\in \mathcal{N}_2$ then $\sigma$ has multiplicity $m(\sigma)=2$ and its eigenspace in \eqref{spect2} is induced by the basis
            $$\mathcal{B}_\sigma=\{(1, \ldots, 1), \quad (\Psi, \ldots, \Psi)\},$$
            where $1$ and $\Psi$ are two linear independent eigenfunctions of $\sigma$ in \eqref{eign.1.tree2}.
        \item If $\lambda\not \in \mathcal{N}_2$ then $\sigma$  has multiplicity $m(\sigma)=1$ and a basis for its eigenspace in \eqref{spect2} is
            $$\mathcal{B}_\sigma=\{(1, \ldots, 1)\}.$$
      \end{itemize}
  \item If $\sigma<0$ then we have the alternative:
  \begin{enumerate}
    \item either $\sigma$ is eigenvalue of \eqref{eign.1.tree2}, situation in which we distinguish two cases: 
    \begin{itemize}
    	\item if $\lambda\in \mathcal{N}_1$,
     $\sigma$ as an eigenvalue of \eqref{spect2} has multiplicity $m(\sigma)=2$ and a basis for its  eigenspace   is given by
        $$\mathcal{B}_\sigma=\{(\Psi, \ldots, \Psi), \quad (\tilde{\Psi}, \ldots, \tilde{\Psi})\},$$
        where $\Psi$ and $\tilde{\Psi}$ are two linear independent eigenfunctions of $\sigma$ in \eqref{eign.1.tree2}.
        \item if $\lambda\notin \mathcal{N}_1$ then $\sigma$ is simple and the basis of its eigenspace is given by      $$\mathcal{B}_\sigma=\{(\Psi, \ldots \Psi)\},$$
where $\Psi$ is an eigenfunction of $\sigma$ in \eqref{eign.1.tree2}.
            \end{itemize}
    \item or $\sigma$ is an eigenvalue of \eqref{eign22}, situation in which we distinguish two cases:
        \begin{itemize}
    \item If $\lambda \in \mathcal{N}_{odd}$ then there exists two linear independent eigenfunctions $\Phi, \tilde{\Phi}$ of $\sigma$ in \eqref{eign22} and there exists two  vectors $\overrightarrow{C}=(c_k)_{k= 1, N}$, $\overrightarrow{D}= (d_k)_{k=1, N}$ such that $\sum_{k=1}^{N}c_k=\sum_{k=1}^{N}d_k=0$ and
        $$\phi=\overrightarrow{C} \Phi+ \overrightarrow{D} \tilde{\Phi}.$$
        In this case $\sigma$ has multiplicity $m(\sigma)=2(N-1)$ and a basis for its eigenspace is given by
        $$\mathcal{B}_\sigma=\{\Phi{\bf e}_{l}-\Phi{\bf e}_{l+1},\quad \tilde{\Phi}{\bf e}_{l}-\tilde{\Phi}{\bf e}_{l+1} \}_{l=1, N-1}.$$
          \item If $\lambda\not \in \mathcal{N}_{odd}$  there exists an eigenpair $(\sigma, \Phi)$ of \eqref{eign22} and  there exists a nontrivial  vector $\overrightarrow{C}=(c_k)_{k=1, N}$ with $\sum_{k=1}^{N}c_k=0$ such that
        $$\phi=\overrightarrow{C} \Phi.$$ In this case  $\sigma$ has multiplicity $m(\sigma)=N-1$ and a basis for its eigenspace is given by
        $$\mathcal{B}_{\sigma}=\{\Phi{\bf e}_{l}-\Phi{\bf e}_{l+1}\}_{l=1, N-1}.$$
        \end{itemize}
  \end{enumerate}
  \end{enumerate}
\end{lemma}
\begin{remark}
  The analysis in the preliminary sections \ref{prelim1} and \ref{prelim2} allows us to say more about the eigenfunctions $\Psi, \tilde{\Psi}, \Phi, \tilde{\Phi}$ in Lemma \ref{eigen_partition3} since they are actually sinus or cosinus type functions.
\end{remark}
\begin{proof}[Proof of Lemma \ref{eigen_partition3}]
First we prove the partition of the eigenvalues \eqref{equiparti}.  The fact that $\sigma(E_1)\cap \sigma(E_2)=\emptyset$ is a consequence of Lemma \ref{not.common2}. 
Now, assume that $(\sigma, \phi=(\phi^k)_{k=1, N})$ is an eigenpair of \eqref{spect2}. Then $\sigma$ verifies \eqref{eign.1.tree2} for $\Psi=S$ in \eqref{functie_suma2}. In addition, $\sigma$ verifies \eqref{eign22} for any $\Phi=D^k$ in \eqref{functie_diferenta2}.
If $S\not\equiv 0$ then ($\sigma$, $S$) is an eigenpair of \eqref{eign.1.tree2}. Otherwise, if $S\equiv 0$, according to Lemma \ref{not.common2} we have $D^k=\phi^k$ for all $k\in \{1, \ldots, N\}$. Consequently, there exists $k_0 \in \{1, \ldots, N\}$ such that $\phi^{k_0}\neq 0$ and therefore $(\sigma, \phi^{k_0})$ is an eigenpair for \eqref{eign22}.

Conversely,  let $(\sigma, \Psi)$ be an eigenpair for \eqref{eign.1.tree2}. Then $(\sigma, \phi=(\Psi, \ldots, \Psi, \Psi))$ is an eigenpair for \eqref{spect2}. Let $(\sigma, \Phi)$ be an eigenpair for \eqref{eign22}. Then $(\sigma, \phi=(0, \ldots,0, -\Phi, \Phi))$ is an eigenpair for \eqref{spect2}, which completes the first part of Lemma \ref{eigen_partition3}.

The rest of the proof follows in each one of the cases $\sigma>0$, $\sigma=0$ and $\sigma<0$ from the preliminary analysis in sections \ref{prelim1} and \ref{prelim2}.
\end{proof}
The previous results allow us to conclude the following lemma.
\begin{lemma}\label{neqzero2}
For any $\la \not \in \mathcal{N}_{mixt}$  any eigenfunction $\phi=(\phi^k)_{k=1, N}$ of $A$ in \eqref{op2} satisfies
$\phi^k(L) \neq 0$ for at least two indexes, or $\lambda \phi^k(L)+\phi_{xx}^k(L)\neq 0$ for at least two indexes $k\in \{1, \ldots, N\}$.
\end{lemma}
\proof
With the same notations as above we have that $S$ and $D^k$, $k\in \{1, \ldots, N\}$, satisfy \eqref{eign.1.tree2} and \eqref{eign22}. We distinguish two cases as follows.

  \textit{The case $S\not\equiv 0$}. Using  the first part of Lemma \ref{eigen_partition3} we must have $D^k\equiv 0$, for all $k\in \{1, \ldots, N\}$. This means that $\phi = (S/N, \dots,S/N)$ where $S$ is an eigenfunction of problem \eqref{eign.1.tree2}. So, from Lemma  \ref{lem1} it holds that $S(L)\neq 0$ or $\lambda S(L)+S_{xx}(L)\neq 0$. This gives the desired result.

\textit{The case $S\equiv 0$}. In this case we have $D^k=\phi^k$ for all $k\in\{1, \ldots, N\}$. Assume that for at least $N-1$ indexes $k\in \{1, \ldots, N\}$ we have $\phi^k(L)= 0$ and   also $\lambda \phi^k(L)+\phi_{xx}^k(L)=0$ for at least (possibly different) $N-1$ indexes.  Since $S\equiv 0$ we must have $\phi^k(L)= \phi_{xx}^{k}(L)=0$ for all $k\in \{1, \ldots, N\}$. On the other hand, from the hypothesis we know  that there exists $k_0\in \{1, \ldots, N\}$ such that $\phi^{k_0}\not\equiv 0$.    This implies that $\phi^{k_0}$ is an eigenfunction for \eqref{eign22}. Applying Lemma \ref{lem1} we must have $\phi^{k_0}_{xx}(L)\neq 0$ or $\lambda \phi^{k_0}(L)+\phi_{xx}^{k_0}(L)\neq 0$, which leads to a contradiction.
%
\endproof

In the following we obtain useful spectral asymptotic properties of problem  \eqref{spect2}.
\begin{lemma}\label{l1} Let $\{\sigma_n\}_{n\geq n_0}$ be the family of positive eigenvalues for the operator $A$ in \eqref{op2} and  let $\phi_n\in\mathcal{B}_{\sigma_n}$ be a corresponding eigenfunction of the spectral problem \eqref{spect2} with $\|\phi_n\|_{L^2(\Gamma)}=1$.  We have:
\begin{enumerate}
\item The positive eigenvalues  of problem \eqref{spect2} satisfy the asymptotic property
\[
  \sigma_n=  \left(\frac{\pi}{2L}\right)^4 \left(n-n_0+ \left[\frac{2L \sqrt{\lambda}}{\pi}\right]+1\right) ^4 + o(n^3), \quad n\rightarrow \infty.
\]
\item\label{i1} There exists two positive constant $C^1_{N, L}$ and  $C^2_{N, L}$ depending only on $N$ and $L$ such that
\[|\phi_n^k(L)|\in \{ C^1_{N, L}, C^2_{N,L} \}\]
holds for any  nontrivial components $\phi_n^k$, $k\in \{1, \ldots, N\}$. 
\item\label{i2} There exists two positive constant $D^1_{N, L}$ and  $D^2_{N, L}$ depending only on $N$ and $L$ such that  for any $k\in\{1,\dots,N\}$
\[\lim_{n\rightarrow \infty}\frac{|\lambda\phi_n^k(L)+\phi_{n, xx}^k(L)|}{n^2}\in \{ D_{N,L}^1,D^2_{N, L}\},
\]
where the limit is taken over all the eigenfunctions $(\phi_n^k)_{n\geq n_0}\in B_{\sigma_n}$ whose the $k$-th component does not vanish identically.
\end{enumerate}

\end{lemma}

\begin{proof}[Proof of Lemma \ref{l1}]
First let us make the notation $m_0(\lambda)=\left[\frac{2L \sqrt{\lambda}}{\pi}\right]+1$.
According to sections \ref{prelim1}-\ref{prelim2} and Lemma \ref{eigen_partition3} we have the partition of the positive eigenvalues
\begin{align*}
  \{\sigma_n \ |\ n\geq n_0  \}&=\Big\{\sigma_{1, n} \ | \ 2n>\frac{2L\sqrt \lambda}{\pi} \Big\}\cup \Big\{\sigma_{2, n} \ | \ 2n+1>\frac{2L\sqrt \lambda}{\pi} \Big\}\\
  &=\Big\{ \Big(\frac{m\pi}{2L}\Big)^4 -\lambda \Big(\frac{m\pi}{2L}\Big)^2, \quad m\geq m_0(\lambda) \Big\},
\end{align*}
where $\sigma_{1, n}$ are the positive eigenvalues of \eqref{eigenval} whereas $\sigma_{2, n}$ are the positive eigenvalues of \eqref{eigenval2}.
Let us now define the sequence
$$\tilde{\sigma}_n:=  \Big(\frac{n\pi}{2L}\Big)^4 -\lambda \Big(\frac{n\pi}{2L}\Big)^2, \quad n\geq m_0(\lambda) $$
Finally we remark that
$$\sigma_n=\tilde{\sigma}_{n-n_0+m_0(\lambda)}, \quad n\geq n_0,$$
and the asympotic formula is proved.
  The rest of the proof of Lemma \ref{l1} is a direct consequence of the analysis done in sections \ref{prelim1} and \ref{prelim2}. We omit further details here since both the eigenvalues and eigenfunctions are explicitly determined in sections \ref{prelim1} and \ref{prelim2} and easy computations are just to be checked.
\end{proof}

\subsection{Controllability problem}
The control problem \eqref{control}-(II) is reduced to solve the following moment problem.
Similarly as in Lemma \ref{momentpb} we can show
\begin{lemma}\label{momentpb1}
Let $\{\sigma_n\}_{n\geq 0}$ be the set of distinct eigenvalues of system \eqref{spect2} and denote by $m(\sigma_n)$ the multiplicity of $\sigma_n$ whose eigenspace is generated by linear independent eigenfunctions normalized in $L^2(\Gamma)$, say,     $\{\phi_{n, l}\}_{l=1, m(\sigma_n)}$. System \eqref{control}-(II) is null-controllable if for any initial data $y_0=(y_0^k)_{k=1, N}\in L^2(\Gamma)$,
\[
y_{0} =\sum_{n\in \nn}\sum_{l=1}^{m(\sigma_n)} y_{0, n, l} \phi_{n,l}
\] and any time $T>0$,
there exist controls $u=(a^k, b^k)_{k=1, N}\in (H^1(0,T))^{2N}$ such that
\begin{align}\label{mo1}
y_{0, n, l} e^{_-T \sigma_n} =& \sum_{k=1}^{N}  (\lambda \phi_{n, l}^k(L)+\phi_{n, l, xx}^k(L)) \intt a^k(T-t) e^{-t \sigma_n} \dt\nonumber\\
&\quad + \sum_{k=1}^{N} \phi_n^k (L) \intt b^k(T-t) e^{-t \sigma_n} \dt, \quad \forall n\geq 0, \quad l=1, \dots, m(\sigma_n).
\end{align}
\end{lemma}

\subsection{Proof of Theorem \ref{main theorem 2}.} In view of Lemma \ref{momentpb1} it is sufficient to ensure the existence of controls $u=(a^k, b^k)_{k=1, N}$ satisfying \eqref{mo1}. The construction of such controls is again based on the method of moments of Fattorini-Russell \cite{MR0335014} implemented  in the proof of Theorem \ref{main theorem 1}.

{\bf Step I. } 
First we focus on the first statement in Theorem \ref{main theorem 2}.
Without loss of generality, we assume for simplicity that the control does not act on the $N$-th component, i.e.  $u^N=(a^N, b^N)\equiv (0, 0)$. 
Then the  problem \eqref{mo1} becomes
    \begin{align}\label{mo2}
y_{0, n, l} e^{_-T \sigma_n} =& \sum_{k=1}^{N-1}  (\lambda \phi_{n,l}^k(L)+\phi_{n, l, xx}^k(L)) \intt a^k(T-t) e^{-t \sigma_n} \dt\nonumber\\
&\quad + \sum_{k=1}^{N-1} \phi_{n,l}^k (L) \intt b^k(T-t) e^{-t \sigma_n} \dt, \quad \forall n\geq 0\nonumber, \, \forall \, l=1,\dots, m(\sigma_n).\\
\end{align}

We now explain  how we construct the sequences $\{a^k\}$, $\{b^k\}$, depending if $\sigma$ is eigenvalue of problem \eqref{eign.1.tree2} or \eqref{eign22}. In view of Lemma \ref{equiparti} with have various situations as follows.

Case I. $\sigma_n$ is eigenvalue of \eqref{eign.1.tree2}.

\begin{enumerate}
	\item $\sigma_n>0$, $m(\sigma_n)=1$, $\phi_{n,1}^k=\Psi$, $k=1,\dots, N$.
In view of Lemma \ref{lem1} both terms $ \Psi(L)$ and $\lambda \Psi(L)+\Psi_{xx}(L)$ do not vanish so we can choose
\[
\intt a^k(T-t) e^{-t \sigma_n} \dt=\frac 1{2(N-1)}\frac {y_{0, n, 1} e^{_-T \sigma_n}}{\lambda \Psi(L)+\Psi_{xx}(L)}, \quad k=1,\dots, N-1,
\]
\[
\intt b^k(T-t) e^{-t \sigma_n} \dt=\frac 1{2(N-1)}\frac {y_{0, n, 1} e^{_-T \sigma_n}}{ \Psi(L)},\quad k=1,\dots, N-1.
\]

\medskip

\item $\sigma_n=0$, 
\begin{itemize}
	\item 
$\lambda\in \mathcal{N}_2$, $m(\sigma_n)=2$, $\phi_{n,1}^k=\Psi\equiv 1$, $\phi^k_{n,2}=\tilde \Psi\equiv \cos (\sqrt{\lambda}x) $.
Choosing $\intt a^k(T-t) e^{-t \sigma_n} \dt=A_n$, $\intt b^k(T-t) e^{-t \sigma_n} \dt=B_n$, for all $k=1,\dots, N-1$, it remains to solve the system
\[
\begin{pmatrix}
  \lambda\Psi(L)+\Psi_{xx}(L) & \Psi(L) \\
  \lambda\tilde\Psi(L)+\tilde\Psi_{xx}(L) &\tilde\Psi(L)
\end{pmatrix}
\begin{pmatrix}
 A_n \\[10pt]
B_n
 \end{pmatrix}=
  \frac {e^{-T\sigma_n} }{N-1}
   \begin{pmatrix}
y_{0, n, 1}  \\
y_{0, n, 2}
 \end{pmatrix}.
\]
Computing explicitly the determinant in the left hand side we find that equals $\pm 1$ and the system is compatible.

\item $\lambda\in \mathcal{N}_2$, $m(\sigma_n)=1$,  $\phi_{n,1}^k\equiv 1$, $k=1,\dots, N$. We make the same choice as in the case $\sigma_n>0$ replacing $\Psi$ with constant function $1$.
\end{itemize}

\medskip

\item $\sigma_n<0$,
\begin{itemize}
	\item 
 $\lambda\in \mathcal{N}_1$, $m(\sigma_n)=2$, $\phi_{n,1}^k=\Psi\equiv \cos(\beta_nx)$, $\phi^k_{n,2}=\tilde \Psi\equiv \cos (\beta_mx) $.
With the same choice as in the previous case we obtain that the determinant in the left hand side satisfies
\[
\begin{vmatrix}
  \lambda\Psi(L)+\Psi_{xx}(L) & \Psi(L) \\
  \lambda\tilde\Psi(L)+\tilde\Psi_{xx}(L) &\tilde\Psi(L)
  \end{vmatrix}
  =(\beta_n^2-\beta_m^2)\cos(\beta_n L)\cos(\beta_m L)=(\beta_n^2-\beta_m^2)(-1)^{m+n},
\]
so the system is compatible.

\item $\lambda\notin \mathcal{N}_1$, $m(\sigma_n)=1$, $\phi_{n,1}^k=\Psi=\cos(\beta_n x)$. In view of Lemma \ref{lem1} both terms $ \Psi(L)$ and $\lambda \Psi(L)+\Psi_{xx}(L)$ do not vanish so we can make the choice as in the case $\sigma>0$.
\end{itemize} 
\end{enumerate}

Case II. $\sigma_n$ is eigenvalue of \eqref{eign22}. In this case $\sigma_n\neq 0$ since $\lambda\notin \mathcal{N}_3$.

\begin{enumerate}
	\item $\sigma_n>0$, $m(\sigma_n)=N-1$, and a basis for the associated eigenspace is given by
	$\{\Phi e_l-\Phi e_{l+1}\}_{l=1,N-1}$ where $\Phi$ solves \eqref{eign22}.
	
	In this case our system becomes
	\[
	(\lambda \Phi(L)+\Phi_{xx}(L))M\vec{\bold a}
		+
	\Phi(L) M\vec{\bold b}
		=e^{-T\sigma_n} \vec{\bold y}
		\]
		where $M$ is the matrix introduced in the proof of Theorem \ref{main theorem 1} and 
		\[
		\vec{\bold a}=\begin{pmatrix}
		\intt a^1(T-t) e^{-t \sigma_n} \dt\\
		\dots\\
		\intt a^{N-1}(T-t) e^{-t \sigma_n} \dt
	\end{pmatrix} ,
\vec{\bold b}=\begin{pmatrix}
		\intt b^1(T-t) e^{-t \sigma_n} \dt\\
		\dots\\
		\intt b^{N-1}(T-t) e^{-t \sigma_n} \dt
	\end{pmatrix},
	\vec{\bold y}=\begin{pmatrix}
		y_{0,n,1} \\
		\dots\\
		y _{0,n,N-1}
	\end{pmatrix}.
\]
		
	In view of Lemma \ref{lem2} both terms $ \Phi(L)$ and $\lambda \Phi(L)+\Phi_{xx}(L)$ do not vanish so we choose
	\[
	\vec{\bold a}=\frac{1}{\lambda \Phi(L)+\Phi_{xx}(L)}M^{-1}
e^{-T\sigma_n}
		 \vec{\bold y}, \quad \vec{\bold b}=\vec{0}.
	\]
	We remark that the following choice is also possible
	\[
	\vec{\bold a}=\vec{0}, \quad \vec{\bold b}=\frac{1}{ \Phi(L) }M^{-1}
e^{-T\sigma_n}
		 \vec{\bold y}.
	\]

	\item $\sigma<0$, $\lambda\notin \mathcal{N}_{odd}$, $m(\sigma_n)=N-1$.
	The construction is the same as in the case II (1) above.
	
	\item $\sigma<0$, $\lambda\in \mathcal{N}_{odd}$, $m(\sigma_n)=2(N-1)$ and a basis for the associated eigenspace is given by
	$\{\Phi e_l-\Phi e_{l+1}\}_{l=1,N-1}$, $\{\tilde\Phi e_l-\tilde\Phi e_{l+1}\}_{l=1,N-1}$
	 where $\Phi\equiv \sin(\beta_nx) $,  $\tilde\Phi\equiv \sin(\beta_mx)$ for some $m\neq n$.
	In this case $\vec{\bold a}$ and $\vec{\bold b}$ solve the system
	\[
	\left\{
	\begin{array}{l}
	(\lambda \Phi(L)+\Phi_{xx}(L))M\vec{\bold a}
		+
	\Phi(L) M\vec{\bold b}
		=e^{-T\sigma_n}
		 \vec{\bold y}_1 \\
	(\lambda \tilde\Phi(L)+\tilde\Phi_{xx}(L))M\vec{\bold a}
		+
	\tilde\Phi(L) M\vec{\bold b}
		=e^{-T\sigma_n}
		 \vec{\bold y}_2
	\end{array}
	\right.
	\]
	where 
	$\vec{\bold y}_1=(y_{0,n,1},\dots, y_{0,n,N-1})^T$ and 
	$\vec{\bold y}_2=(y_{0,n,N},\dots, y_{0,n,2(N-1)})^T$. 
	Explicit computations show that
	\[
\begin{vmatrix}
  \lambda\Phi(L)+\Phi_{xx}(L) & \Phi(L) \\
  \lambda\tilde\Phi(L)+\tilde\Phi_{xx}(L) &\tilde\Phi(L)
  \end{vmatrix}
  =(\beta_m^2-\beta_n^2)\sin(\beta_n L)\sin(\beta_m L)=(\beta_m^2-\beta_n^2)(-1)^{m+n}\neq 0
\]
and
\[
\begin{pmatrix}
	\vec{\bold a}\\
	\vec{\bold b}
\end{pmatrix}=
\frac{1}{(\beta_m^2-\beta_n^2)(-1)^{m+n}}
\begin{pmatrix}
	\tilde\Psi(L) M^{-1}& -  \Psi(L) M^{-1}\\
-	( \lambda\tilde\Psi(L)+\tilde\Psi_{xx}(L))M^{-1}&  ( \lambda\Psi(L)+\Psi_{xx}(L) )M^{-1}
\end{pmatrix}
e^{-T\sigma _n}\vec{\bold y},
\]
where 	$ \vec{\bold y}=( \vec{\bold y}_1,  \vec{\bold y}_2)^T.$

 \end{enumerate}
Summarizing, in order to finish the proof of the main part of Theorem \ref{main theorem 2} it is enough to solve a moment problem for each $a^k$, $b^k$, for any $k=1, N-1$, that is
$$\int_{0}^T a^k(T-t)e^{-t\sigma_n} \dt = c_{n, k}, \quad \forall n\in \nn,$$
respectively
$$\int_{0}^T b^k(T-t)e^{-t\sigma_n} \dt = d_{n, k}\quad \forall n \in \nn,$$
for precised sequences $\{c_{n,k}\}_n, \{d_{n,k}\}_n$ determined in the analysis above.
   In consequence, similar as in the proof of Theorem \ref{main theorem 1} in view of Lemma \ref{l1} such moment problems can be solved.
%
%
%
    Then, in view of \cite{MR0335014} we obtain the
  null-controllability for system \eqref{control}-(II).

\medskip
	\textbf{Step II. Optimality of (2N-2) controls.} Let us suppose that we can control with $2N-3$ controls.
	Without loss of generality we can consider the following two cases: $a^{N}\equiv a^{N-1}\equiv b^N\equiv 0$ or $a^{N}\equiv b^{N-1}\equiv b^N\equiv 0$. 
	 In both cases we prove that for $\lambda\in \mathcal{N}_{odd}$ and $\sigma <0$ eigenvalue of \eqref{spect2}  initial data of the type
	\[
	y_0=y_{01}( \Phi e_{N-1}- \Phi e_{N})+y_{02}(\tilde \Phi e_{N-1}-\tilde \Phi e_{N})
	\]
	cannot be driven to the null state where $\Phi(x)=\sin (\beta_n x)$, $\tilde \Phi(x)=\sin (\beta_mx)$, $\lambda=\beta_n^2+\beta_m^2$. We emphasize that, up to normalization, $ \Phi e_{N-1}- \Phi e_{N}$ and $\tilde \Phi e_{N-1}-\tilde \Phi e_{N}$ are elements of the orthonormal basis $\{\phi_{n, l}\}_{l=1, m(\sigma_n), n\geq 0}$ in the hypothesis of Lemma \ref{momentpb1}.
	
	Case I. $a^N\equiv a^{N-1}\equiv b^N\equiv 0$. In this case system \eqref{mo2} becomes
	 \begin{align*}
y_{0 1} e^{_-T \sigma} = \Phi (L) \intt b^{N-1}(T-t) e^{-t \sigma} \dt,\\
y_{0 2} e^{_-T \sigma} =\tilde \Phi (L) \intt b^{N-1}(T-t) e^{-t \sigma} \dt.
\end{align*}
Explicit computations shows that $\Phi(L)=(-1)^n$,  $\tilde\Phi(L)=(-1)^m$. Choosing $y_{01}=0$ and $y_{02}=1$ leads to a contradiction.

Case II. $a^N\equiv b^{N}\equiv b^{N-1}\equiv 0$. In this case system \eqref{mo2} becomes
 \begin{align*}
y_{0 1} e^{_-T \sigma} = (\lambda \Phi (L)+\Phi_{xx}(L)) \intt a^{N-1}(T-t) e^{-t \sigma} \dt,\\
y_{0 2} e^{_-T \sigma} = (\lambda \tilde\Phi (L)+\tilde\Phi_{xx}(L))\intt a^{N-1}(T-t) e^{-t \sigma} \dt.
\end{align*}
Explicit computations shows that $\lambda \Phi (L)+\Phi_{xx}(L)=(-1)^n\beta_n^2$ and $\lambda \tilde\Phi (L)+\tilde\Phi_{xx}(L)=(-1)^m\beta_m^2$. Choosing $y_{01}=0$ and $y_{02}=1$ leads again to a contradiction.

\medskip
\textbf{Step III. Null-controllability with $(2n-3)$ controls.} When  $\lambda\notin \mathcal{N}_{odd} $ we can easily adapt the proof given in the Step I in order to construct the controls. The details are left to the reader.

%
%
%

\section{Further control  results}
\subsection{Null-controllability of systems \eqref{control}-(I) and (II)}\label{secf}
In the fist part of the paper we have been concerned with studying controllability problems acting with a minimal number of control inputs.
However, we have to mention that we can also address the question for which values $\lambda>0$  systems \eqref{control}-(I) and \eqref{control}-(II) are null-controllable for a maximal number of control inputs. In fact, we are able to prove that system \eqref{control}-(II) is null-controllable for any $\lambda\notin \mathcal{N}_3$ if we act with $2N$ controls. In contrast with that, system \eqref{control}-(I) is not null-controllable for  $\lambda\in  \mathcal{N}_{1}\cup \mathcal{N}_{odd}$ even if we  act with $N$  controls. More precisely we obtain

\begin{theorem} [Null-controllability for model \eqref{control}-(I)]\label{maxcontrol_I}
Let  $T>0$.
 \begin{enumerate}
 \item\label{m3}  For any $\lambda \not \in \mathcal{N}_{1}\cup \mathcal{N}_{odd}$ and  $y_0=(y_0^k)_{k=1, N} \in L^2(\Gamma)$,   there exist  controls $u=(u^k)_{k=1, N}\in (H^1(0, T))^N$  such that the solution of system \eqref{control}-(I) satisfies
\begin{equation}\label{nullcontrol1}
y^k(T, x)=0, \textrm{  for any } x\in (0, L)\ \text{and} \ k\in \{1, \ldots, N\}.
\end{equation}
\item\label{m4}  Assume  $\lambda \in \mathcal{N}_{1}\cup \mathcal{N}_{odd}$. There exist initial state $y_0=(y_0^k)_{k=1, N} \in L^2(\Gamma)$ such that for any control $u=(u^k)_{k=1, N}\in (H^1(0, T))^N$  the solution of system \eqref{control}-(I) satisfies
    $$y^{k_0}(T, \cdot)\not  \equiv 0,$$
    for some $k_0\in \{1, \ldots, N\}$.
\end{enumerate}
\end{theorem}
\begin{theorem}[Null-controllability for model \eqref{control}-(II)]\label{maxcontrol_II}
Let $T>0$.
\begin{enumerate}
	\item 
 For any $\lambda\not \in \mathcal{N}_3$ and $y_0=(y_0^k)_{k=1, N}\in L^2(\Gamma)$  there exist  controls $a=(a^k)_{k=1, N}, b=(b^k)_{k=1, N} \in (H^1(0, T))^N$ such that the solution of system \eqref{control}-(II) satisfies
\begin{equation}\label{nullcontrol2}
y^k(T, x)=0, \textrm{  for any } x\in (0, L), \quad k\in \{1, \ldots, N\}.
\end{equation}
\item For   $\lambda \in \mathcal{N}_{3}$ there exist an initial state $y_0=(y_0^k)_{k=1, N} \in L^2(\Gamma)$ such that for any controls $a=(a^k)_{k=1, N}, b=(b^k)_{k=1, N}\in (H^1(0, T))^N$   the solution of system \eqref{control}-(II) satisfies
    $$y^{k_0}(T, \cdot)\not  \equiv 0,$$
    for some $k_0\in \{1, \ldots, N\}$.
\end{enumerate}
\end{theorem}

 Here we avoid  the details of the proof of Theorem \ref{maxcontrol_II}  since they follow similar ideas and steps as the main theorems proved in the paper.  The details are let to the reader.

Next we sketch the proof of Theorem \ref{maxcontrol_I}.

The main ingredient in the proof of Theorem \ref{maxcontrol_I} is the following proposition. 

\begin{proposition}\label{max_prop1}
Let $\lambda>0$ and $(\sigma, \phi=(\phi^k)_{1, N})$ be an eigenpair of problem \eqref{spect}. With the notations in Lemma \ref{eigen_partition} we have the alternative:
\begin{enumerate}
	\item\label{m1} If $(\sigma, \Psi)$ is eigenpair of \eqref{eign.1.tree} then 
	   \begin{itemize}
	   	\item If $\sigma \geq 0$ then $\Psi_{xx}(L)\neq 0$ for all $\lambda>0$
	   	\item If $\sigma<0$ then $\Psi_{xx}(L)\neq 0$ iff $\lambda\not\in \mathcal{N}_{odd}$
	   \end{itemize}
	\item\label{m2} If $(\sigma,\Phi)$ is eigenpair of  \eqref{eign2} then 
   \begin{itemize}
	   	\item If $\sigma \geq 0$ then $\Phi_{xx}(L)\neq 0$ for all $\lambda>0$
	   	\item If $\sigma<0$ then $\Phi_{xx}(L)\neq 0$ iff $\lambda\not\in \mathcal{N}_{1}$
	   \end{itemize}
\end{enumerate}
\end{proposition}


It is then obvious that the proof of \eqref{m3} in Theorem \ref{maxcontrol_I} is a trivial consequence of the first part of Theorem \ref{main theorem 1}. The statement \eqref{m4} in Theorem \ref{maxcontrol_I} is a consequence of  Proposition \ref{max_prop1}. Indeed, if $(\sigma, \Psi)$ is an eigenpair of \eqref{eign.1.tree} then the initial state $y_0=(\Psi, \ldots, \Psi)$ cannot be leaded to the zero state when $\lambda\in \mathcal{N}_{odd}$ in view of \eqref{ort}  and the fact $\Psi_{xx}(L)=0$. If $(\sigma, \Phi)$ is an eigenpair of \eqref{eign2} then the initial state $y_0=(\Phi, -\Phi, 0, \ldots, 0)$ cannot be leaded to the zero state when $\lambda\in \mathcal{N}_{1}$ in view of \eqref{mo1}  and the fact $\Phi_{xx}(L)=0$.

\begin{proof}[Proof of Proposition \ref{max_prop1}]\textit{The case $(\sigma, \Psi)$ is an eigenpair of \eqref{eign.1.tree}.} 
Assume that $\Psi_{xx}(L)=0$. Then $\Psi$ in satisfies the overdetermined problem 
\begin{equation}\label{overdet_sum_1}
\left\{
\begin{array}{ll}
\lambda \Psi_{xx}+\Psi_{xxxx}=\sigma \Psi, & x\in (0,L),\\[10pt]
 \Psi_x(0)= \Psi_{xxx}(0)=0 , &\\[10pt]
 \Psi(L)=\Phi_x(L)=\Psi_{xx}(L)=0 &\\[10pt]
\end{array}
\right.
\end{equation}
 Next we distinguish several cases in terms of the sign of $\sigma$.

\noindent{Case $\sigma >0$.} As in section \ref{prelim1} the boundary conditions at $x=0$ lead to
$$\Psi(x)=C_1\cosh (\alpha x)+C_3 \cos (\beta x).$$
Imposing the conditions at $x=L$ in \eqref{overdet_sum_1} we get
\be\label{matrix1}
\mathfrak{M} \left(\begin{array}{c}
    C_1 \\
    C_3
  \end{array}\right)=\left(\begin{array}{c}
                       0 \\
                       0 \\
                       0
                     \end{array}\right)
\ee
where $$ \mathfrak{M} =\left(\begin{array}{cc}
            \cosh(\alpha L) & \cos(\beta L) \\
            \alpha\sinh(\alpha L) & -\beta \sin(\beta L) \\
            \alpha^2 \cosh(\alpha L) & -\beta^2 \cos(\beta L)
          \end{array}\right).
$$
Observe that $\textrm{rank } \mathfrak{M} =2$ since $\cos(\beta L)$ and $\sin (\beta L)$ cannot vanish simultaneously.  Therefore $\Psi\equiv 0$, which is a contradiction with the fact that $\Psi$ is an eigenfunction. 

\noindent{Case $\sigma=0$}. From section \ref{prelim1} and the conditions at $x=0$ we necessary have
$$\Psi(x)=C_1+C_3 \cos (\sqrt{\lambda }x).$$
From the conditions at $x=L$ we get
\be\label{matrix1}
\mathfrak{P} \left(\begin{array}{c}
    C_1 \\
    C_3
  \end{array}\right)=\left(\begin{array}{c}
                       0 \\
                       0 \\
                       0
                     \end{array}\right)
\ee
where $$\mathfrak{P}=\left(\begin{array}{cc}
            1 & \cos(\sqrt{\lambda} L) \\
            0 & -\sqrt{\lambda} \sin(\sqrt{\lambda} L) \\
            0 & -\lambda \cos(\sqrt{\lambda} L)
          \end{array}\right).
$$
Again, $\textrm{rank } \mathfrak{P}=2 $ since $\cos(\sqrt{\lambda}L)$ and $\sin(\sqrt{\lambda}L)$ cannot vanish simultaneously.  This implies $C_1=C_3=0$ and therefore $\Psi\equiv 0$ which is a contradiction. 

\noindent{Case $\sigma<0$.} From section \ref{prelim1} and \eqref{overdet_sum_1} we necessary have
$$\Psi(x)=C_1 \cos (\gamma x)+ C_3 \cos(\beta x).$$
 Then from the conditions at $x=L$ in \eqref{overdet_dif_1} we obtain
 \be\label{matrix1}
\mathfrak{R} \left(\begin{array}{c}
    C_1 \\
    C_3
  \end{array}\right)=\left(\begin{array}{c}
                       0 \\
                       0 \\
                       0
                     \end{array}\right)
\ee
where $$\mathfrak{R}=\left(\begin{array}{cc}
            \cos(\gamma L) & \cos(\beta L) \\
            -\gamma \sin(\gamma L) & -\beta\sin(\beta L) \\
            -\gamma^2 \cos(\gamma L) & -\beta^2 \cos(\beta L)
          \end{array}\right).
$$
It is easy to see that $\textrm{rank } \mathfrak{R} =1$ if and only if $\cos(\gamma L)=\cos(\beta L)=0$ which is equivalent to $\lambda\in \mathcal{N}_{odd}$. In other words, $\textrm{ rank } \mathfrak{R} =2$ if and only if $\lambda \not \in \mathcal{N}_{odd}$ in which case we get $\Psi\equiv 0$. Contradiction.

Therefore $\Psi_{xx}(L)\neq 0$ in all the situations stated in Proposition \ref{max_prop1}. 
  If $\sigma<0$ and $\lambda\not \in \mathcal{N}_{odd}$ then we obtain $\cos(\gamma L)=\cos(\beta L)=0$ which implies from above that $\Psi_{xx}(L)=0$. 

Thus the proof of \eqref{m1} is finished. 

\textit{The case  $(\sigma,\Phi)$ is eigenpair of  \eqref{eign2}}. We proceed as in the previous case: for the proof of \eqref{m2} we first assume that $\Phi_{xx}(L)=0$ and then $\Phi$
 verifies
\begin{equation}\label{overdet_dif_1}
\left\{
\begin{array}{ll}
\lambda \Phi_{xx}+D_{xxxx}=\sigma \Phi, & x\in (0,L),\\[10pt]
\Phi(0)=\Phi_{xx}(0)=0 , &\\[10pt]
\Phi(L)=\Phi_{x}(L)=\Phi_{xx}(L)=0, &\\[10pt]
\end{array}
\right.
\end{equation}
Case $\sigma>0$.  From section \ref{prelim2} imposing the boundary conditions at the origin we have
$$\Phi(x)=C_2 \sinh(\alpha x)+C_4 \sin (\beta x).$$
The conditions at $x=L$ in \eqref{overdet_dif_1} give
\be\label{matrix1}
\mathfrak{N}\left(\begin{array}{c}
	C_2 \\
	C_4
\end{array}\right)=\left(\begin{array}{c}
	0 \\
	0 \\
	0
\end{array}\right)
\ee
where $$\mathfrak{N}=\left(\begin{array}{cc}
\sinh(\alpha L) & \sin(\beta L) \\
\alpha\cosh(\alpha L) & \beta \cos(\beta L) \\
\alpha^2 \sinh(\alpha L) & -\beta^2 \sin(\beta L)
\end{array}\right).
$$
Since $\textrm{rank } \mathfrak{N} =2$ we obtain $D\equiv 0$ which is in contradiction with the fact that $\phi$ is an eigenfunction.

{Case $\sigma=0$}.    Going back to section \ref{prelim2} from  the first conditions in \eqref{overdet_dif_1} we get   that
$$D(x)= C_2 x +C_4 \sin (\sqrt{\lambda }x ).$$
Applying the conditions at $x=L$ in \eqref{overdet_dif_1} we obtain
\be\label{matrix1}
\mathfrak{Q} \left(\begin{array}{c}
	C_2 \\
	C_4
\end{array}\right)=\left(\begin{array}{c}
	0 \\
	0 \\
	0
\end{array}\right)
\ee
where $$\mathfrak{Q}=\left(\begin{array}{cc}
L & \sin(\sqrt{\lambda} L) \\
1 & \sqrt{\lambda} \cos(\sqrt{\lambda} L) \\
0 & -\lambda \sin(\sqrt{\lambda} L)
\end{array}\right).
$$
Since $\textrm{rank } \mathfrak{Q}=2$ we obtain $\Phi\equiv 0$. Contradiction.

{Case $\sigma<0$}.  Again, in view of section \ref{prelim2} we  have
$$\Phi(x)=C_2 \sin (\gamma x)+C_4 \sin (\beta x).$$
Imposing the conditions at $x=L$ in \eqref{overdet_dif_1} we obtain
\be\label{matrix1}
\mathfrak{S} \left(\begin{array}{c}
	C_2 \\
	C_4
\end{array}\right)=\left(\begin{array}{c}
	0 \\
	0 \\
	0
\end{array}\right)
\ee
where $$\mathfrak{S}=\left(\begin{array}{cc}
\sin(\gamma L) & \sin(\beta L) \\
\gamma \cos(\gamma L) & \beta\cos(\beta L) \\
-\gamma^2 \sin(\gamma L) & -\beta^2 \sin(\beta L)
\end{array}\right).
$$
We easily deduce that $\textrm{rank } \mathfrak{S}=1$ if and only if $\sin (\gamma L)=\sin(\beta L)=0$ which is equivalent to $\lambda \in \mathcal{N}_1$. Therefore, $\textrm{rank } \mathfrak{S}=2$ if and only if $\lambda \not \in \mathcal{N}_1$ in which case we get $\Phi\equiv 0$. Contradiction.

Therefore $\Phi_{xx}(L)\neq 0$ in all the situations stated in Proposition \ref{max_prop1}. 
If $\sigma<0$ and $\lambda\not \in \mathcal{N}_{1}$ then we obtain $\sin(\gamma L)=\sin(\beta L)=0$ which implies from above that $\Phi_{xx}(L)=0$. 

This ensures the proof of \eqref{m2}  and thus, the proof of Proposition \ref{max_prop1} is finished.
\end{proof}

%

\subsection{New control results for the linear KS on an interval}\label{secf2}

In this section we present some new control results for a single linear KS equation  which are direct consequences of the spectral analysis developed in Section \ref{secII}. For the sake of clarity we will not insist too much on the rigorousness of the technical details since we already did it in the previous sections.

We consider the system
\be\label{singleKS}
\left\{\begin{array}{ll}
 y_t +\la y_{xx}+y_{xxxx}=0, & (t, x) \in (0, T)\times (0,L)\\[10pt]
y_{x}(t, 0)=u^1(t), & t\in (0, T) \\[10pt]
y_{xxx}(t, 0)=u^2(t),  & t\in (0, T) \\[10pt]
y_x(t, L)=y_{xxx}(t, L)=0, & t\in (0, T)\\[10pt]
y(0, x)=y_0(x), & x\in (0, L).\\[10pt]
\end{array}\right.
\ee
where $u^1, u^2$ are control inputs. We then obtain

\begin{theorem}\label{theo1.1}
The system \eqref{singleKS} is null-controllable for any  $\lambda >0$: for any time $T>0$  and any initial data $y_0\in L^2(0, L)$  there exist two controls $u^1, u^2\in H^1(0, T)$ which steer the solution of \eqref{singleKS} to the zero state, i.e. $y(T, x)=0$, for all $x\in (0, L)$.
\end{theorem}

Roughly speaking, the null-controllability of \eqref{singleKS} is given through the adjoint problem (backward in time)

\be\label{adjoint_singleKS}
\left\{\begin{array}{ll}
 -q_t +\la q_{xx}+q_{xxxx}=0, & (t, x) \in (0, T)\times (0,L)\\[10pt]
q_{x}(t, 0)=q_{xxx}(t, 0)=0,  & t\in (0, T) \\[10pt]
q_x(t, L)=q_{xxx}(t, L)=0, & t\in (0, T)\\[10pt]
q(T, x)=q_T(x), & x\in (0, L).\\[10pt]
\end{array}\right.
\ee

As in subsection \ref{contPb} the system \eqref{singleKS} is null-controllable if one ensures the existence of  the control inputs $u^1, u^2$ such that
\be\label{sing1}
\int_0^L y_0(x)q(0, x) dx = \int_0^T u^1(t) (\lambda q(t, 0) + q_{xx}(t, 0))+\int_0^T u^2(t)q(t, 0) dt.
\ee

Note that the  spectral problem corresponding to the adjoint system \eqref{adjoint_singleKS} is precisely the eigenvalue problem \eqref{eigenval} in subsection \ref{prelim1}.

Therefore, in view of the method of moments such controls $u^1, u^1$ satisfying \eqref{sing1} can be build if  any eigenfunction $\phi$ of system \eqref{eigenval}  verifies  $\phi(0)$ and $\lambda \phi(0) + \phi_{xx}(0)$  cannot vanish simultaneously. This is true as a consequence of Lemma \ref{lem1} applied at the point $x=0$ instead of $x=L$ (we let the details to the reader to check that Lemma \ref{lem1} is also valid when replacing $L$ with 0).

\medskip
Another direct application of our spectral results in Section \ref{secII} regards the following system

\be\label{singleKS2}
\left\{\begin{array}{ll}
 y_t +\la y_{xx}+y_{xxxx}=0, & (t, x) \in (0, T)\times (0,L)\\[10pt]
y(t, 0)=u^1(t), & t\in (0, T) \\[10pt]
y_{xx}(t, 0)=u^2(t),  & t\in (0, T) \\[10pt]
y_x(t, L)=y_{xxx}(t, L)=0, & t\in (0, T)\\[10pt]
y(0, x)=y_0(x), & x\in (0, L).\\[10pt]
\end{array}\right.
\ee
whose adjoint is given by

\be\label{adjoint_singleKS2}
\left\{\begin{array}{ll}
 -q_t +\la q_{xx}+q_{xxxx}=0, & (t, x) \in (0, T)\times (0,L)\\[10pt]
q(t, 0)=q_{xx}(t, 0)=0,  & t\in (0, T) \\[10pt]
q_x(t, L)=q_{xxx}(t, L)=0, & t\in (0, T)\\[10pt]
q(T, x)=q_T(x), & x\in (0, L).\\[10pt]
\end{array}\right.
\ee
We easily observe that the spectral problem corresponding to system \eqref{adjoint_singleKS2} is nothing else than the eigenvalue problem \eqref{eigenval2} in subsection \ref{prelim2}.
By the above considerations system \eqref{singleKS2} is null-controllable if one can find $u^1, u^2$ such that
\be\label{sing2}
\int_0^L y_0(x)q(0, x) dx + \int_0^T u^1(t) (\lambda q_x(t, 0) + q_{xxx}(t, 0))+\int_0^T u^2(t)q_x(t, 0) dt=0.
\ee
This is equivalent to verify if each eigenfunction $\phi$ of \eqref{eigenval2} satisfies that $\phi_x(0)$ and $\lambda \phi_x(0) + \phi_{xxx}(0)$ do not vanish simultaneously. Indeed, as a consequence of the complete determination of the eigenfunctions of \eqref{eigenval2} in subsection \ref{prelim2}, this is true for any $\lambda>0$. Therefore, we obtain the following controllability result. 
\begin{theorem}\label{theo2}
The system \eqref{singleKS2} is null-controllable for any  $\lambda >0$: for any time $T>0$  and any initial data $y_0\in L^2(0, L)$  there exist two controls $u^1, u^2\in H^1(0, L)$ which steer the solution of \eqref{singleKS2} to the zero state, i.e. $y(T, x)=0$, for all $x\in (0, L)$.
\end{theorem}

\medskip

\subsection*{Acknowledgments} C. C. was partially supported by  a Young Researchers Grant awarded by The Research Institute of the University of Bucharest (ICUB) and also by the CNCS-UEFISCDI Grant No. PN-III-P4-ID-PCE-2016-0035.  L. I. was  partially supported by  a grant of the Romanian National Authority for Scientific Research and Innovation, CNCS--UEFISCDI, project number PN-II-RU-TE-2014-4-0007.
 A. P. was partially supported by CNPq (Brasil) and Agence Universitaire de la Francophonie. This work started when A. P.  visited the ``Simion Stoilow" Institute of Mathematics of the Romanian Academy in Bucharest and continued with the visit of C. C.  at Instituto de Matem\'{a}tica from  Universidade Federal do Rio de Janeiro in Brasil.  The authors thank both institutions for their hospitality and the possibility to work in a good research atmosphere. Finally, we wish to express our gratitude to the referees for their remarks and suggestions which helped us to  consistently improve the initial submission of the paper.

\bibliographystyle{amsplain}
\bibliography{bib_cip_7}
\end{document}